\documentclass[a4paper]{amsart}

\usepackage{esint}
\usepackage{MnSymbol}
\usepackage{enumitem}
\usepackage[english]{babel}

\usepackage{tikz}
\usetikzlibrary{calc, shapes, backgrounds, patterns}

\usepackage[
	colorlinks,
	pdfpagelabels,
	pdfstartview = FitH,
	bookmarksopen = true,
	bookmarksnumbered = true,
	linkcolor = blue,
	plainpages = false,
	hypertexnames = false,
	citecolor = red] {hyperref}

\hypersetup{
	linktoc=page
}

\usepackage[capitalize]{cleveref}
\crefname{enumi}{}{}
\crefname{equation}{}{}

\makeatletter
\def\@tocline#1#2#3#4#5#6#7{\relax
  \ifnum #1>\c@tocdepth 
  \else
    \par \addpenalty\@secpenalty\addvspace{#2}%
    \begingroup \hyphenpenalty\@M
    \@ifempty{#4}{%
      \@tempdima\csname r@tocindent\number#1\endcsname\relax
    }{%
      \@tempdima#4\relax
    }%
    \parindent\z@ \leftskip#3\relax \advance\leftskip\@tempdima\relax
    \rightskip\@pnumwidth plus4em \parfillskip-\@pnumwidth
    #5\leavevmode\hskip-\@tempdima
      \ifcase #1
       \or\or \hskip 1em \or \hskip 2em \else \hskip 3em \fi%
      #6\nobreak\relax
    \dotfill\hbox to\@pnumwidth{\@tocpagenum{#7}}\par
    \nobreak
    \endgroup
  \fi}
\makeatother

\newtheorem{theorem}{Theorem}[section]
\newtheorem{proposition}[theorem]{Proposition}
\newtheorem{lemma}[theorem]{Lemma}

\theoremstyle{definition}
\newtheorem{definition}[theorem]{Definition}
\newtheorem{remark}[theorem]{Remark}

\numberwithin{equation}{section}

\def \R {{\mathbb {R}}}

\def\supp{\operatorname{supp}}

\def\diam{\operatorname{diam}}
\def\grad{\nabla}
\renewcommand{\tilde}{\widetilde}
\newcommand{\dx}{\, {\rm d} x}
\newcommand{\dt}{\, {\rm d} t}
\newcommand{\de}{\, {\rm d}}
\renewcommand{\fint}{\strokedint}

\allowdisplaybreaks

\begin{document}
	

\title[Boundary Behavior]{Boundary behavior of solutions to the\\ parabolic $p$-Laplace equation}

\author[Avelin]{Benny Avelin}
\address{
	Benny Avelin,
	Department of Mathematics and Systems Analysis,
	Aalto University School of Science,
	FI-00076 Aalto,
	Finland
}
\address{
	Benny Avelin,
	Department of Mathematics,
	Uppsala University,
	S-751 06 Uppsala,
	Sweden
}
\email{\color{blue} benny.avelin@math.uu.se}

\author[Kuusi]{Tuomo Kuusi}
\address{
	Tuomo Kuusi,
	Department of Mathematics and Systems Analysis,
	Aalto University School of Science,
	FI-00076 Aalto,
	Finland
}
\email{\color{blue} tuomo.kuusi@aalto.fi}

\author[Nystr\"om]{Kaj Nystr\"om}
\address{
	Kaj Nystr\"{o}m,
	Department of Mathematics,
	Uppsala University,
	S-751 06 Uppsala,
	Sweden
}
\email{\color{blue} kaj.nystrom@math.uu.se}

\date{\today}



\begin{abstract}
	We establish boundary estimates for non-negative solutions to the $p$-parabolic equation in the degenerate range $p>2$. Our main results include new parabolic intrinsic Harnack chains in cylindrical NTA-domains together with sharp boundary decay estimates. If the underlying domain is $C^{1,1}$-regular, we establish a relatively complete theory of the boundary behavior, including boundary Harnack principles and H\"older continuity of the ratios of two solutions, as well as fine properties of associated boundary measures. There is an intrinsic waiting time phenomena present which plays a fundamental role throughout the paper. In particular, conditions on these waiting times rule out well-known examples of explicit solutions violating the boundary Harnack principle.
\end{abstract}
\maketitle


\section{Introduction and results}

This paper is devoted to a study of the boundary behavior of non-negative solutions to the $p$-parabolic equation, in the degenerate range $p>2$. We restrict the analysis to space-time cylinders $\Omega_T=\Omega\times (0,T)$, $T>0$, where $\Omega\subset\R^n$ is a bounded domain, i.e., an open and connected set. Given $p$, $1< p<\infty$, fixed, recall that the $p$-parabolic equation is the equation
\begin{equation} \label{basic eq}
	 \partial_t u - \Delta_p u := \partial_tu-\nabla\cdot (|\nabla u|^{p-2}\nabla u) = 0\,.
\end{equation}

In the special case $p=2$  the $p$-parabolic equation coincides with the heat equation, and in this case we refer to Kemper~\cite{K}, Salsa~\cite{S}, and also~\cite{FGSit,FGS,FS,FSY,G,N}, concerning the boundary behavior of non-negative solutions. Key results established in these works, in the context of Lipschitz-cylinders $\Omega_T$, include Carleson type estimates, the relation between the associate parabolic measure and Green function, the backward in time Harnack inequality, boundary Harnack principles (local and global) and  H\"older continuity up to the boundary of quotients of non-negative solutions vanishing on the lateral boundary.

On the contrary for $p\neq 2$, $1<p<\infty$, much less is known concerning these problems and we refer the reader to~\cite{A,AGS, KMN} for accounts of the current literature. For a relatively complete picture in the case of non-linear parabolic operators with linear growth we refer to \cite{NPS}. However, it is also important to mention that there is an interesting and related recent literature devoted to the asymptotic and pointwise behavior of solutions to non-linear diffusion equations on bounded domains, see \cite{SV}, and also \cite{BoVa} for the porous medium type equations,  and the references therein.

Considering non-negative solutions to the $p$-parabolic equation, for $p$ in the degenerate range $p>2$, it is a priori not clear to what extent and in what sense the above mentioned results can remain to hold. Indeed, on the one hand we have to account for the lack of homogeneity of the $p$-parabolic equation, and on the other hand we have to account for the fact that in the degenerate regime the phenomenon of finite speed propagation is present. As a matter of fact, simple examples show that in this case there are, compared to the case $p=2$, much more delicate waiting time phenomena to take into account.

To discuss the aspects of the waiting time phenomena further, we here first briefly describe some by now classical results in the case $p=2$, see~\cite{FGS,S}. Assume that $\Omega$ is, say, a Lipschitz domain, that $x_0 \in \partial \Omega$ and let $A_\pm := (a_r(x_0),t_0 \pm r^2)$, where $a_r(x_0)$ is an interior point of $\Omega$ with distance to the boundary comparable to $r$.
Assume also that $u$ and $v$ are non-negative caloric functions in $\Omega_T$, i.e., functions satisfying~\cref{basic eq} with $p=2$ in $\Omega_T$, vanishing continuously on $(\partial \Omega \cap B_r(x_0)) \times (t_0-r^2, t_0 + r^2)$, where $ B_r(x_0)\subset\mathbb R^n$ is the standard Euclidean ball of radius $r$ and centered at $x_0\in\mathbb R^n$. Then
\begin{equation}
	\label{st1} c^{-1}\frac{u(A_-)}{v(A_+)}\leq \frac{u(x,t)}{v(x,t)}\leq c\frac{u(A_+)}{v(A_-)}\,,
\end{equation}
 for a universal constant $c$, whenever $(x,t) \in (\Omega \cap B_{r/2}(x_0)) \times (t_0-(r/2)^2, t_0 + (r/2)^2)$. However, in general an estimate like~\cref{st1} dramatically fails in the case $p \neq 2$. To see this, recall the following two classical solutions (see e.g.~\cite{AGS}) in the case when $\Omega := \R^{n-1} \times \{x_n : x_n>0\}$:
\begin{equation}
	\label{st1+} u(x,t)= c_p (T-t)^{-1/(p-2)}x_n^{p/(p-2)} \,, \qquad v(x,t) = x_n\,.
\end{equation}

In view of the examples in \cref{st1+} it is not clear under what conditions the boundary Harnack principle in~\cref{st1} could hold. Let us make a few observations. When defining $u$ as in~\cref{st1+}, we see that the larger we take $T$, the longer the solution $u$ exists and the smaller its pointwise values become at a fixed time $t < T$.
If we wish to show an estimate as in~\cref{st1}, we need to be able to rule out examples like $u$  in \cref{st1+} (see also some other examples in \cite{AGS}). We do this by simply requiring, for  $(x,t) \in \Omega_T$ fixed, that
\begin{equation} \label{st1++}
	T-t > C_0 u(x,t)^{2-p} d(x,\partial \Omega)^p\,,
\end{equation}
for a large enough constant $C_0$. It is easily seen that the solution $u$ in~\cref{st1+} does not satisfy~\cref{st1++} at any point $(x,t) \in \Omega_T$ if we require $C_0 \geq c_p^{p-2}$.

In this context, and for the $p$-parabolic equation, it is here natural to make a link to the by now  classical method of {\em intrinsic scaling} due to DiBenedetto, see~\cite{DB}. The intrinsic scalings define the canonical geometry in which weak solutions to the $p$-parabolic equation become homogenized in a sense to be made precise. Indeed, in this geometry we consider, instead of the standard parabolic cylinders, intrinsically time-scaled cylinders of the type
\begin{align*}
	Q_r^{\lambda,+}(x,t) &:= B_r(x) \times (t, t +  \lambda^{2-p} r^p)\,, \\
	Q_r^{\lambda,-}(x,t) &:= B_r(x) \times (t - \lambda^{2-p} r^p, t)\,, \qquad \lambda := u(x,t)\,.
\end{align*}
These kinds of intrinsic cylinders appear naturally in the context of Harnack inequalities, oscillation reduction estimates, and decay estimates, and defines the correct geometry in our setting.

The main goal of this paper is to study to what extent the theory developed in \cite{FGS,S} generalizes to the case $p>2$, under suitable intrinsic conditions. We have already seen that it rules out the pathological examples like $u$ in~\cref{st1+}. In fact, we prove that~\cref{st1++} is a sufficient condition for developing a rather general theory concerning the boundary behavior of non-negative solutions to the $p$-parabolic equation. For instance,  ~\cref{st1++} allows us to prove a counterpart of~\cref{st1} valid for $2<p<\infty$, see~\Cref{theo3+} below.


\subsection{Summary of results}

We will now give an informal summary of our results. The precise statements can be found in the bulk of the paper.

\subsubsection*{Harnack chains}
 Fundamental tools in the study of the boundary behavior of non-negative solutions are the Harnack inequality and Harnack chains. Harnack chains allow one to relate the value of non-negative solutions at  different space-time points in the domain. For  $p=2$ the Harnack inequality is homogeneous and, roughly, to control the values of the solution in a ball of size $r$ requires a waiting time comparable to $r^2$. For  $p > 2$ we have to use an intrinsic version of the Harnack inequality \cite{DB,DGV}. In particular the intrinsic Harnack inequality states, see \Cref{Har}, that if we have a non-negative solution $u$ to the $p$-parabolic equation in $\Omega_T$, with $(x,t) \in \Omega_T$, and $Q_{4r}^{u(x,t)/c_h ,-}(x,t) \subset \Omega_T$, then
\begin{equation*}
	u(x,t) \leq C_h \inf_{y \in B_r(x)} u\left (y, t + \left [\frac{c_h}{u(x,t)} \right ]^{p-2} r^p \right)
\end{equation*}
provided that $t + \left [\frac{c_h}{u(x,t)} \right ]^{p-2} r^p < T$. The intrinsic waiting time required in this  Harnack inequality is consistent with the condition stated in \Cref{st1++}. In~\Cref{s.Harnack} we develop a sequence of Harnack chain estimates and the goal of that section is twofold. First, we want to establish estimates applicable in cylindrical NTA-domains, see~\Cref{NTA}. Second, we want to establish a $p$-stable counterpart of the sharp Harnack chain estimate proved by Salsa~\cite[Theorem C]{S}, which in the case $p=2$ reads as
\begin{equation} \label{Salsa Harnack}
		u(y,s) \leq u(x,t) \exp\left [ C \left ( \frac{|x-y|^2}{t-s} + \frac{t-s}{k} + 1\right )\right ]\,,
\end{equation}
where $k = \min\{1,s,d(x,\partial \Omega)^2, d(y,\partial \Omega)^2\}$, $s < t$, and $(x,t),(y,s) \in \Omega_T$. To do this we develop Harnack chains based on the weak Harnack inequality proved in~\cite{Ku}, see~\Cref{TuomoWH} below, valid for supersolutions to the $p$-parabolic equation. As truncations of our solutions are supersolutions to \Cref{basic eq}, we are able to control the waiting times more precisely by adjusting the levels at which the solutions are truncated. This is in sharp contrast to the Harnack chains developed in \cite{A,AGS} for which there is very little control over the waiting time. Our approach to Harnack chains has at least three advantages. First, it allows us to construct Harnack chains starting from the measure $u(x,t_0)\dx$ at the initial time $t_0$. Second, it allows us to develop a flexible Carleson estimate, see \Cref{s.Carleson},  generalizing the one in \cite{AGS} and which, in addition, remains valid in the context of time-independent NTA-cylinders. We note that although the Carleson estimate proved in \cite{A} is valid in the setting of  time-independent NTA-cylinders, a difference compared to the results in this paper is that  the Carleson estimate proved in \cite{A} is not $p$-stable as $p\to 2$. Third, we develop a version of~\cref{Salsa Harnack}  which is $p$-stable in the sense that we recover~\cref{Salsa Harnack} as $p \to 2$. To establish this Gaussian type behavior for $p>2$ is technically rather involved.

\subsubsection*{Estimates of associated boundary measures}
In the study of the boundary behavior of quasi-linear equations of $p$-Laplace type, certain Riesz measures supported on the boundary and associated to non-negative solutions vanishing on a portion of the boundary are important, see \cite{LN1,LN2}. These measures are non-linear generalizations of the harmonic measure relevant in the study of the Laplace equation and the Green function. Corresponding  measures can also be associated to solutions to the $p$-parabolic equation. Indeed, let $u$ be a non-negative solution in $\Omega_T$, assume that $u$ is continuous on the closure of $\Omega_T$, 
and that $u$ vanishes on $\partial_p \Omega_T \cap Q$ with some open set $Q$. Extending $u$ to be zero in $Q \setminus \Omega_T$, 
it is straightforward to see that $u$ is a continuous weak subsolution to~\cref{basic eq} in $Q$. Using this, one can conclude that there exists a unique locally finite positive Borel measure $ \mu$, supported on $S_T \cap  Q$, such that
\begin{align}
	\label{1.1+} & \int_Q u \partial_t\phi\dx \dt -\int_Q|\nabla u|^{p-2}\nabla u\cdot\nabla \phi \dx \dt = \int_Q \phi \de \mu
\end{align}
whenever $ \phi \in C_0^\infty(Q)$. In \Cref{s.measure} we establish, in cylindrical NTA-domains, both upper and lower bounds for the measure $\mu$ in terms of $u$. If $\Omega$ is smooth, then  $\de \mu=|\grad u|^{p-1}\de H^{n-1} \dt$.  Based on this the lower bound established on $\mu$ can be interpreted as a non-degeneracy estimate, close to the boundary, of the solution. Our proof of the lower bound for the measure $\mu$ is a modification of the elliptic proof, see for example \cite{ANsub,KZ}. However, our proof is genuinely non-linear, it applies to much more general operators of $p$-parabolic type, and the result seems to be new already in the case $p=2$.

\subsubsection*{A ``complete theory'' in $C^{1,1}$-domains} 
We establish a  ``complete theory'' concerning the boundary behavior of non-negative solutions in $\Omega_T$ in the case when $\Omega$  is a  $C^{1,1}$-domain. As a comprehensive literature is missing, we in~\Cref{s.global,s.local} develop both a local as well as a global theory of boundary behavior in $C^{1,1}$-cylinders. In the global setting we are able, as in~\cite{FGS} with corresponding estimates in the case $p=2$, to give a rather complete picture. For nonnegative solutions vanishing on the lateral boundary, our results include a global boundary Harnack principle and H\"older continuity of ratios. On the other hand, in the local setting we prove a new intrinsic local boundary Harnack principle.  In the context of $C^{1,1}$-cylinders we are also able to show that the boundary measure in~\cref{1.1+} is mutually absolutely continuous with respect to the surface measure in a suitably chosen intrinsic geometry. The results in~\Cref{s.global,s.local} are obtained by combining Harnack chains and Carleson estimates with explicit barrier constructions from~\Cref{s.barriers} and decay estimates from~\Cref{s.decay}.

\vspace{0.5cm}

\noindent {\bf Acknowledgment} The first author was supported by the Swedish Research Council, dnr: 637-2014-6822. The second author was supported by the Academy of Finland \#258000.


\section{Notation and preliminaries}
Points in $ \R^{n+1} $ are denoted by $ x = ( x_1, \dots, x_n,t)$. Given a set $E\subset \R^n$, let $ \bar E, \partial E$, $\mbox{diam } E$, $E^c$, $E^\circ $, denote the closure, boundary, diameter, complement and interior of $E$, respectively. Let $ \cdot $ denote the standard inner product on $ \R^{n} $, let $ | x | = (x \cdot x )^{1/2}$ be the Euclidean norm of $ x, $ and let $\dx $ be Lebesgue $n$-measure on $ \R^{n}. $ Given $ x \in \R^{n}$ and $r >0$, let $ B_{r} (x) = \{ y \in \R^{n} : | x - y | < r \}$. Given $ E, F \subset \R^{n}, $ let $ d ( E, F ) $ be the Euclidean distance from $ E $ to $ F $. In case $ E = \{y\}, $ we write $ d ( y, F )$. For simplicity, we define $\sup$ to be the essential supremum and $\inf$ to be the essential infimum. If $ O \subset \R^{n} $ is open and $ 1 \leq q \leq \infty, $ then by $ W^{1 ,q} ( O )$ we denote the space of equivalence classes of functions $ f $ with distributional gradient $ \nabla f = ( f_{x_1}, \dots, f_{x_n} ), $ both of which are $ q $-th power integrable on $ O. $ Let \,
\[ \| f \|_{ W^{1,q} (O)} = \| f \|_{ L^q (O)} + \| \, | \nabla f | \, \|_{ L^q ( O )} \, \]
be the norm in $ W^{1,q} ( O ) $ where $ \| \cdot \|_{L^q ( O )} $ denotes the usual Lebesgue $q$-norm in $O$. $ C^\infty_0 (O )$ is the set of infinitely differentiable functions with compact support in $ O$ and we let $ W^{1 ,q}_0 ( O )$ denote the closure of $ C^\infty_0 (O )$ in the norm $\| \cdot\|_{ W^{1,q} (O)}$. $ W^{1,q}_{\rm loc} ( O ) $ is defined in the standard way. By $ \nabla \cdot $ we denote the divergence operator. Given $t_1<t_2$ we denote by $L^q(t_1,t_2,W^{1,q} ( O ))$ the space of functions such that for almost every $t$, $t_1\leq t\leq t_2$, the function $x\to u(x,t)$ belongs to $W^{1,q} ( O )$ and
\begin{equation*}
	\| u \|_{ L^q(t_1,t_2,W^{1,q} ( O ))}:=\biggl (\int\limits_{t_1}^{t_2}\int\limits_O\biggl (|u(x,t)|^q+|\nabla u(x,t)|^q\biggr )\dx\dt\biggr )^{1/q} <\infty\,.
\end{equation*}
The spaces $L^q(t_1,t_2,W^{1,q}_0 ( O ))$ and $L^q_{\rm loc}(t_1,t_2,W^{1,q}_{\rm loc} ( O ))$ are defined analogously. Finally, for $I \subset \R$, we denote $C(I;L^{q} ( O ))$ as the space of functions such that $t\to \| u(t,\cdot) \|_{ L^{q} (O)}$ is continuous whenever $t \in I$. $C_{\rm loc}(I;L^{q}_{\rm loc} ( O ))$ is defined analogously.


\subsection{Weak solutions} Let $ \Omega \subset \R^n $ be a bounded domain, i.e., a connected open set. For $t_1<t_2$, we let $\Omega_{t_1,t_2}:= \Omega \times (t_1,t_2)$. Given $p$, $1<p<\infty$, we say that $ u $ is a weak solution to
\begin{equation} \label{Hu}
	\partial_t u - \Delta_p u = 0
\end{equation}
in $ \Omega_{t_1,t_2}$ if $ u \in L_{\rm loc}^p(t_1,t_2,W_{\rm loc}^{1,p} ( \Omega))$ and
\begin{equation}
	\label{1.1tr} \int_{\Omega_{t_1,t_2} } \left(- u \partial_t \phi  + |\nabla u|^{p-2}\nabla u\cdot\nabla \phi \right) \dx\dt = 0
\end{equation}
whenever $ \phi \in C_0^\infty(\Omega_{t_1,t_2})$. If $ u $ is a weak solution to~\Cref{Hu} in the above sense, then we will often refer to $u$ as being $p$-parabolic in $\Omega_{t_1,t_2}$. For $p \in (2,\infty)$ we have by the parabolic regularity theory, see \cite{DB}, that any $p$-parabolic function $u$ has a locally H{\"o}lder continuous representative. In particular, in the following we will assume that $p \in (2,\infty)$ and any solution $u$ is continuous. If \Cref{1.1tr} holds with $=$ replaced by $\geq$ ($\leq$) for all $ \phi \in C_0^\infty(\Omega_{t_1,t_2})$, $\phi \geq 0$, then we will refer to $u$ as a weak supersolution (subsolution).


\subsection{Geometry} We here state the geometrical notions used throughout the paper.
\begin{definition}
	\label{NTA} A bounded domain $\Omega$ is called non-tangentially accessible (NTA) if there exist $M \geq 2$ and $r_0$ such that the following are fulfilled:
	\begin{enumerate}[label=\bf{(\arabic*)}]
		\item \label{NTA1} \emph{corkscrew condition:} for any $ w\in \partial\Omega, 0<r<r_0,$ there exists a point $a_r(w) \in \Omega $ such that
		\begin{equation}
			\nonumber M^{-1}r<|a_r(w)-w|<r, \quad d(a_r(w), \partial\Omega)>M^{-1}r \,,
		\end{equation}
		\item \label{NTA2} $\R^n \setminus \Omega$ satisfies \cref{NTA1},
		\item \label{NTA3} \emph{uniform condition:} if $ w \in \partial \Omega, 0 < r < r_0, $ and $ w_1, w_2 \in B_r(w) \cap \Omega, $ then there exists a rectifiable curve $ \gamma: [0, 1] \to \Omega $ with $ \gamma ( 0 ) = w_1,\, \gamma ( 1 ) = w_2, $ such that
		\begin{enumerate}
			\item $H^1 ( \gamma ) \, \leq \, M \, | w_1 - w_2 |,$
			\item $\min\{H^1(\gamma([0,t])), \, H^1(\gamma([t,1]))\, \}\, \leq \, M \, d ( \gamma(t), \partial \Omega)$, for all $t \in [0,1]$.
		\end{enumerate}
	\end{enumerate}
\end{definition}

We choose this definition as it very useful when we explicitly construct the parabolic Harnack chains in \Cref{s.Harnack}, see specifically \Cref{NTA FChainMeas}.
The values $ M $ and $r_0$ will be called the NTA-constants of $ \Omega$. For more on the notion of NTA-domains we refer to \cite{JK}.

\begin{definition}
	\label{defBall}
	Let $\Omega \subset \R^n$ be a bounded domain. We say that $\Omega$ satisfies the ball condition with radius $r_0 > 0$ if for each point $y \in \partial \Omega$ there exists points $x^+ \in \Omega$ and $x^- \in \Omega^c$ such that $B_{r_0}(x^+) \subset \Omega$, $B_{r_0}(x^-) \subset \Omega^c$, $\partial B_{r_0}(x^+) \cap \partial \Omega =\{y\}= \partial B_{r_0}(x^-) \cap \partial \Omega$, and such that the points $x^+(y),x^-(y),y$ are collinear for each $y \in \partial \Omega$.
\end{definition}
\begin{remark} \label{ballNTAcork}
	It is easy to see that a domain satisfying  the ball condition with radius $r_0 > 0$ is an NTA-domain with a constant $M$ and $r_0$. In particular, we may canonically choose
	$$
		a_r(x_0) := x_0 + \frac r2 \frac{x^+-x_0}{|x^+-x_0|}\,,
	$$
	since the direction given by $\frac{x^+-x_0}{|x^+-x_0|}$ is unique. The exterior corkscrew point is defined analogously.
\end{remark}
\begin{remark}
	Let $\Omega \subset \R^n$ be a bounded domain. Then $\Omega$ is a $C^{1,1}$ domain if and only if it satisfies the ball condition. For a proof of this fact, see for example \cite[Lemma 2.2]{AKSZ}.
\end{remark}


\subsection{The continuous Dirichlet problem} Assuming that $\Omega$ is a bounded NTA-domain one can prove, see \cite{BBGP} and \cite{KL}, that all points on the parabolic boundary
\begin{equation*}
	\partial_p\Omega_T = S_T \cup (\bar \Omega \times \{0\})\,, \qquad S_T = \partial \Omega \times [0,T]\,,
\end{equation*}
of the cylinder $\Omega_T$ are regular for the Dirichlet problem for the equation \Cref{Hu}. In particular, for any $f\in C( \partial_p\Omega_T)$, there exists a unique Perron-solution $u=u_f^{\Omega_T}\in C(\overline \Omega_T)$ to the Dirichlet problem $\partial_t u - \Delta_p u = 0$ in $\Omega_T$ and $u =f$ on $\partial_p \Omega_T$.


\section{Harnack chains} \label{s.Harnack}

 In this section we prove a sequence of results concerning intrinsic Harnack chains. Forward-in-time chains describe the diffusion with an appropriate waiting time. On the other hand, backward-in-time chains says that if the solution has existed for long enough time, the future values will control the values from the past as well. Throughout the section we let $\Omega \subset \R^n$ be a bounded domain and given $T>0$ we let $\Omega_T=\Omega\times (0,T)$. 


\subsection{Local Harnack inequalities}
We here collect two estimates from the literature.  The following theorem can be found in~\cite{DB,DGV1,DGV}.
\begin{theorem}
	\label{Har} Let $u$ be a non-negative $p$-parabolic function in $\Omega_T$, let $(x_0,t_0)\in\Omega_T$ and assume that $u(x_0,t_0)>0$. There exist positive constants $c_h$ and $C_h$, depending only on $p,n$, such that if $B_{4 r }(x_0)\subset\Omega$ and
	\begin{align*}
		\left (t_0 - \theta (4 r )^p, t_0 + \theta(4 r )^p \right ) \subset (0,T]\,,
	\end{align*}
	where $\theta = \left (\frac{c_h}{u(x_0,t_0)}\right )^{p-2}$,  then
	\begin{equation*}
		u(x_0,t_0)\leq C_h \inf_{B_{ r }(x_0)} u(\cdot,t_0+\theta  r ^p)\,.
	\end{equation*}
	The constants $c_h$ and $C_h$ are stable as $p\to2$ and deteriorate as $p\to\infty$ in the sense that $c_h(p),\ C_h(p)\to\infty$ as $p\to\infty$.
\end{theorem}

The next theorem is instead valid for non-negative weak supersolutions. For the proof we refer to~\cite{Ku}.

\begin{theorem}
	\label{TuomoWH} Let $u$ be a non-negative weak supersolution in $B_{4 r }(x_0) \times (0,T)$. There exist constants $C_i \equiv C_i(p,n)$, $i=1,2,$ such that
	\begin{equation*}
		\fint_{B_{ r }(x_0)} u(x,t_1) \dx \leq \frac 1 2\left ( \frac{C_1  r ^p}{T - t_1} \right )^{\frac{1}{p-2}} + C_2 \inf_{Q} u
	\end{equation*}
	for almost every $0 < t_1 < T$, where $Q := B_{2  r }(x_0) \times (t_1+T_1/2,t_1+T_1)$, and
	\begin{equation*}
		T_1 = \min \left \{ T-t_1, C_1  r ^p \left ( \fint_{B_{ r }(x_0)} u(x,t_1) \dx \right )^{2-p} \right \}.
	\end{equation*}
	In particular, if $T_1 < T-t_1$, then
	\begin{equation*}
		\fint_{B_{ r }(x_0)} u(x,t_1)\dx \leq 2 C_2 \inf_{Q} u\,.
	\end{equation*}
\end{theorem}


\subsection{Forward Harnack chains} \label{ss.fchain}

We begin by describing a simple Harnack chain for weak supersolutions.
\begin{lemma}
	[Weak forward Harnack chains] \label{FChainMeas} Let $\Omega\subset\R^n$ be a domain and let $T>0$. Let $x,y$ be two points in $\Omega$ and assume that there exist a sequence of balls $\{B_{4r}(x_j)\}_{j=0}^k$ such that $x_0=x$, $x_k=y$, $B_{4r}(x_j)\subset\Omega$ for all $j=0,...,k$ and that $x_{j+1} \in B_{r}(x_j)$, $j=0,\ldots,k-1$. Assume that $u$ is a continuous non-negative weak supersolution in $\Omega_T$ with
	\begin{equation*}
		\bar\Lambda:=\fint_{B_{r}(x_0)} u(x,t_0) \dx > 0\,.
	\end{equation*}
	There exist constants $\bar c_i \equiv \bar c_i(p,n) >1$, $i \in \{1,2\}$, such that if
	\begin{align*}
		&t _0 + \tau_k \bar\Lambda^{2-p} r^p< T\,, \qquad \tau_k := \bar c_1 \sum_{j=0}^{k} \bar c_2^{j(p-2)} \,,
	\end{align*}
	then
	\begin{equation*}
		\fint_{B_{r}(x)} u(x,t_0) \dx \leq \bar c_2^{k+1} \inf_{z \in B_{2r}(y)} u(z, t_0 + \tau_k \bar\Lambda^{2-p} r^p)\,.
	\end{equation*}
	Furthermore, constants $\bar c_i$, $i \in \{1,2\}$, are stable as $p \to 2^+$. In particular, when $p=2$, then $\tau_k = \bar c_1(k+1)$ with $\bar c_1 = \bar c_1(n)$.
\end{lemma}
\begin{proof}
	Using \Cref{TuomoWH} we first get that
	\begin{equation} \label{WF initial setp}
		 \frac{\bar\Lambda}{2C_2} \leq \inf_{z \in B_{2r}(x_0)} u(z,t_1) \,, \qquad t_1 := t_0 + C_1 \bar\Lambda^{2-p} r^p\,.
	\end{equation}
	Define then
	\begin{equation*}
		 u_j := \min(u, \Lambda_j)\,, \qquad \Lambda_j := \left(2C_2\right)^{-j} \bar \Lambda \,, \qquad t_{j+1}:= t_{j} + C_1 \Lambda_j^{2-p} r^p \,,
	\end{equation*}
	for $j=\{1,\ldots,k-1\}$.
	Assume inductively, that for $t_{i+1} \leq T$ we have 
	\begin{equation*}
		B_{r}(x_{i}) \subset B_{2r}(x_{i-1})\,,
	\end{equation*}
	and
	\begin{equation*}
		 u_i(z,t_i) = \Lambda_i \qquad \mbox{for } \; z\in B_{r}(x_{i})\,,
	\end{equation*}
	hold for $i \in \{0,\ldots,j\}$. For $j=1$ this is certainly the case as we see from \Cref{WF initial setp}. Since $u_j$ is a non-negative weak supersolution, \Cref{TuomoWH} gives us
	\begin{equation*}
		\Lambda_{j+1} = \frac{\Lambda_j}{2C_2} \leq \inf_{z \in B_{2r}(x_j)} u_j(z,t_{j+1})\,,
	\end{equation*}
	and hence also
	\begin{equation*}
		 u_{j+1}(z,t_{j+1}) = \Lambda_{j+1} \qquad \mbox{for } \; z\in B_{r}(x_{j+1}) \,.
	\end{equation*}
	This proves the induction step. By the construction
	\begin{equation*}
		 \inf_{z \in B_{r}(y)} u_{k}(z,t_{k}) = \Lambda_k
	\end{equation*}
	holds. Thus, applying \Cref{TuomoWH} one more time we get
	\begin{equation*}
		 \inf_{z \in B_{2r}(y)} u(z, \bar t) \geq \left(2C_2\right)^{-(k+1)} \bar \Lambda\,,
	\end{equation*}
	with
	\begin{equation*}
		 \bar t := t_ 0 + C_1 \sum_{j=0}^{k} \left(2C_2\right)^{j(p-2)} \bar \Lambda^{2-p} r^p \,.
	\end{equation*}
	Setting $\bar c_1 = C_1$ and $\bar c_2= 2 C_2$ completes the proof of the lemma.
\end{proof}

For $p$-parabolic functions we have the following pointwise version of \Cref{FChainMeas}.

\begin{proposition}
	Let $\Omega\subset\R^n$ be a domain and let $T>0$. Let $x,y$ be two points in $\Omega$ and assume that there exist a sequence of balls $\{B_{4r}(x_j)\}_{j=0}^k$ such that $x_0=x$, $x_k=y$, $B_{4r}(x_j)\subset\Omega$ for all $j=0,...,k$ and that $x_{j+1} \in B_{r}(x_j)$, $j=0,\ldots,k-1$. Assume that $u$ is a non-negative $p$-parabolic function in $\Omega_T$ and assume that $u(x,t_0)>0$. There exist constants $c\equiv c(p,n)$ and $c_1\equiv c_1(p,n,k)>1$ such that if
	\begin{equation*}
		t_0 - (c_h/u(x,t_0))^{p-2} (4r)^p > 0,\ t_0 + c_1(k) u(x,t_0)^{2-p} r^p <T\,,
	\end{equation*}
	then
	\begin{equation*}
		u(x,t_0) \leq c^{k} \inf_{z \in B_r(y)} u(z,t_0+c_1(k) u(x,t_0)^{2-p} r^p)\,.
	\end{equation*}
	Furthermore, $c_1$ satisfies the estimate
	\begin{equation*}
		\tilde c_1 k \leq c_1 \leq \tilde c_1 (k+1) c^{(k+1)(p-2)}
	\end{equation*}
	with $\tilde c_1 \equiv \tilde c_1(p,n)$ and $\tilde c_1(p,n) \to \tilde c_1(n)$ as $p \to 2$.
\end{proposition}
\begin{proof}
	After applying \Cref{Har} once, the result follows from \Cref{FChainMeas}.
\end{proof}

We next focus on cylindrical NTA-domains. The first theorem, \Cref{NTA FChainMeas}, holds for weak supersolutions and shows how to bound the values of a solution at points close to the boundary using pointwise interior values. A remarkable fact of the proof is that the waiting time is explicitly defined and, as $p\to2$, it gives a supersolution version of the theorem of Salsa~\cite[Theorem C]{S} as alluded to in the introduction, see \Cref{Salsa Harnack}. The proof uses heavily the assumptions on NTA-domains and iterations of~\Cref{FChainMeas}.

\begin{theorem}
	\label{NTA FChainMeas} Let $\Omega\subset\R^n$ be an NTA-domain with constants $M$ and $r_0$, let $x_0 \in \partial \Omega$, $T>0$ and let $0 < r < r_0$. Let $x,y$ be two points in $\Omega \cap B_r(x_0)$ such that
	\begin{equation*}
		\varrho := d(x,\partial \Omega) \leq r \qquad \mbox{and} \qquad d(y,\partial \Omega) \geq \frac{r}{4}\,.
	\end{equation*}
	Assume that $u$ is a non-negative continuous weak supersolution in $\Omega_T$, and assume that
	\begin{equation} \label{eq:NTAFChainMeas}
		\Lambda:=\fint_{B_{\varrho/4}(x)} u(z,t_0) \de z > 0\,.
	\end{equation}
	Let $\delta \in (0,1]$. Then there exist positive constants $c_i \equiv c_i(M,p,n)$, $i\in\{1,2,3\}$, such that if $t_0 + \tau < T $, where
	\begin{equation*}
		\tau := \delta^{p-1} \left( c_2^{-1/\delta} \left( \frac{r}{\varrho} \right)^{- c_3/\delta} \Lambda \right)^{2-p} r^p \,,
	\end{equation*}
	then
	\begin{equation*}
		\fint_{B_{\varrho/4}(x)} u(z,t_0) \de z  \leq c_1^{1/\delta} \left( \frac{r}{\varrho} \right)^{c_3/\delta} \inf_{z \in B_{r/16}(y)} u(z ,t_0 + \tau)\,.
	\end{equation*}
	Furthermore, constants $c_i$, $i \in \{1,2,3\}$, are stable as $p \to 2^+$.
\end{theorem}

\begin{proof}
	We split the proof into three steps.
	\subsubsection*{Step 1: Parametrization of the curve connecting $x$ and $y$} According to the uniform condition \cref{NTA3} in \Cref{NTA}, we can find a rectifiable curve $\gamma$ connecting $x$ and $y$ such that $ \gamma ( 0 ) = x$, $\gamma ( 1 ) = y$, and
	\begin{enumerate}[label=\bf{({\alph*})}]
		\item \label{WFC a} $H^1 ( \gamma ) \, \leq \, M \, | w_1 - w_2 |,$
		\item \label{WFC b} $\min\{H^1(\gamma([0,t])), \, H^1(\gamma([t,1]))\, \}\, \leq \, M \, d ( \gamma(t), \partial \Omega)$, for all $t \in [0,1]$.
	\end{enumerate}
	We call a ball $B \subset \Omega$ \textbf{admissible} if $4B \subset \Omega$, and is thus eligible for the Harnack inequality.
 	Our goal with with step is to construct a sequence of admissible balls covering the curve $\gamma$. In the following we may, without loss of generality, assume that $H^1(\gamma([0,1]))>2^{-4} r$. We define $\hat t_1,\hat t_2 \in (0,1)$ such that $H^1(\gamma([0,\hat t_1])) = 2^{-5} r$ and $H^1(\gamma([\hat t_2,1])) = 2^{-5} r$. The technical part will be in the interval $(0,\hat t_1)$. To continue we choose $k$ as the integer which satisfies $2^{-k} r \in (\varrho/16,\varrho/8]$. We define sequence of real number $\{s_j\}$ through
	\begin{equation*}
		 H^1 ( \gamma([0,s_j]) ) = 2^{-k + j } H^1 ( \gamma([0,\hat t_1]) ):= 2^{-k-5+j} r
		 \,, \qquad s_0 = 0\,.
	\end{equation*}
	Then, for any $s \in [s_{j},s_{j+1})$, \Cref{WFC a,WFC b} implies
	\begin{equation*}
		 d ( \gamma ( s), \partial \Omega ) \geq \frac{2^{-k+j-5}r}{M}
	\end{equation*}
	and
	\begin{equation*}
		 H^1( \gamma ([s_j, s_{j+1} ] )) = 2^{-k+j-5}r\,.
	\end{equation*}
	Thus, defining
	\begin{equation*}
		 \varrho_j := N^{-1} 2^{-k+j-5} r \,, \qquad N \in \mathbb{N}\,, \qquad N \geq 2^9 M\,,
	\end{equation*}
	we see that the piece $\gamma([s_j,s_{j+1}])$ can be covered with $N$ admissible balls of the type $B^{i,j} := B_{\varrho_j}(y_{i,j})$ such that $y_{i,j} \in \gamma([s_j,s_{j+1}])$ for $i \in \{1,\ldots,N\}$, $y_{i,j-1}\in B^{i,j}$ and $\gamma([s_j,s_{j+1}]) \subset \cup_i B^{i,j}$. Finally, we observe that the middle piece of the curve $\gamma([\hat t_1,\hat t_2])$, due to the definition of $\hat t_1,\hat t_2$ together with \Cref{WFC b} can be covered with $M\,N$ admissible balls of size $r/N$. Moreover we can cover the end piece $\gamma([\hat t_2,1])$ with $N$ admissible balls of size $r/N$ since $\gamma([\hat t_2, 1]) \subset B_{r/16}(y) \subset \Omega$. At this point, we consider $N \in \mathbb{N}$ to be a free parameter such that $N \geq 2^9 M$.
	
	\subsubsection*{Step 2: Iteration via Harnack estimates} Let now $\Lambda$ be as in \Cref{eq:NTAFChainMeas}. 	
	\Cref{TuomoWH} implies that if
	\begin{equation*}
		 t_0 + C_1\Lambda^{2-p} \varrho_0^p <T\,,
	\end{equation*}
	then
	\begin{equation*}
		 \inf_{z \in B_{\varrho_0}(x_0)} u(z, t_1) \geq \frac{1}{2 C_2} \Lambda \,, \quad t_1 := t_0 +C_1\Lambda^{2-p} \varrho_0^p\,.
	\end{equation*}
	Let
	\begin{equation*}
		 \Lambda_1 = \sigma \Lambda \,, \qquad \sigma \in (0,(2 C_2)^{-1}]\,.
	\end{equation*}
	Defining thus $u_1 := \min (u,\Lambda_1)$, we obtain by \Cref{FChainMeas} (see also the proof of that Lemma) that there exist constants $\bar c_1 \equiv \bar c_1(p,n)$ and $\bar c_2 \equiv \bar c_2(p,n)$, such that if
	\begin{equation} \label{eq:tauN estimate}
		 t_2 := t_1 + \tau_{N} \Lambda_1^{2-p} \varrho_1^p < T\,, \quad \tau_{N} := \bar c_1 \sum_{j=0}^{N-1} \bar c_2^{j(p-2)} \in [\bar c_1 N, \bar c_1 N \bar c_2^{N(p-2)}) \,,
	\end{equation}
	then
	\begin{equation*}
		 \inf_{z \in B_{\varrho_2}(\gamma(s_2))} u(z,t_2) = \inf_{z \in B_{2\varrho_1}(\gamma(s_2))} u(z,t_2) \geq \frac{\Lambda_1}{\bar c_2^{N}} =: \Lambda_2\,.
	\end{equation*}
	Define
	\begin{equation*}
		 \Lambda_{j+1} := \bar c_2^{-(j+1) N} \Lambda_1\,,
		 \qquad u_j := \min(u,\Lambda_j)\,,\qquad j \in \mathbb{N}\,,
	\end{equation*}
	let $\hat k := k+M+1$, and let
	\begin{align*}
		 t_{j+1} :=
		 \begin{cases}
		 	t_{j} +\tau_{N} \Lambda_{j}^{2-p} \varrho_{j}^p &\text{ if } j \in \{1,\ldots,k\}\\
		 	t_{j} +\tau_{N} \Lambda_{j}^{2-p} (r/N)^p  &\text{ if }j \in \{k+1,\ldots,\hat k\}  \,.
		 \end{cases}
	\end{align*}
	Iterating \Cref{FChainMeas}  it follows by induction  that
	\begin{equation} \label{eq:NTA FCM final bound}
		 u(y,t_{\hat k+1}) \geq \Lambda_{\hat k+1}\,.
	\end{equation}
	
	\subsubsection*{Step 3: Waiting time} Let us now analyze the waiting time $t_{\hat k+1}$, which we want to show to be precisely $t_0+\tau$ by a suitable choice of $\Lambda_1$. We have
	\begin{align}
		t_{\hat k+1} & = t_1 + C_1\Lambda^{2-p} \varrho_0^p+\tau_{N} \left [\sum_{j=1}^{k} \Lambda_{j}^{2-p} \varrho_{j}^p + \sum_{j=k+1}^{\hat k} \Lambda_{j}^{2-p} \varrho_{j}^p\right ] \notag\\
		& = t_0 + C_1\Lambda^{2-p} \varrho_0^p + N^{-p} \tau_N \Lambda_1^{2-p} r^p \bigg [2^{-kp}  \sum_{j=0}^{k} \left( 2^{p} \bar c_2^{(p-2)N} \right)^{j} \label{eq:badsum}\\
		&\qquad +  \bar c_2^{(p-2)kN}\sum_{j=0}^{M} \left( \bar c_2^{(p-2)N} \right)^{j} \bigg ]\label{eq:goodsum} \\
		&=: t_0 + \mathbf{T} \Lambda^{2-p} r^p \notag \,.
	\end{align}
	We can write the sum in \Cref{eq:badsum} as follows
	\begin{align*}
		2^{-kp} \sum_{j=0}^{k}\left( 2^{p} \bar c_2^{(p-2)N} \right)^{j} = 2^p \frac{\bar c_2^{(p-2) (k+1)N} - 2^{-(k+1)p}}{2^{p} \bar c_2^{(p-2)N}-1 }\,,
	\end{align*}
	while the sum in \Cref{eq:goodsum} can be estimated similarly to $\tau_N$ (see \Cref{eq:tauN estimate})
	\begin{equation*}
		\bar c_2^{(p-2)kN}\sum_{j=0}^{M} \left( \bar c_2^{(p-2)N} \right)^{j} \in \left (M \bar c_2^{(p-2)kN} , M \bar c_2^{(p-2)(k+M)N}\right ] \,.
	\end{equation*}
	Hence, recalling the definition of $\tau_N$ and $\mathbf{T}$, we get after some straightforward estimation that
	\begin{align} \label{forward fundamental inequality}
		\frac12 \bar c_1 \sigma^{2-p} N^{1-p} \left( \frac{r}{\varrho} \right)^{(p-2)\bar c_4 N} \leq \mathbf{T} \leq
		2 \bar c_3 \sigma^{2-p} N^{1-p} \bar c_3^{(p-2)N} \left( \frac{r}{\varrho} \right)^{(p-2)\bar c_4 N}
	\end{align}
	for new constants $\bar c_3, \bar c_4$ depending only on $p,n,M$. We now choose $N = \tilde c / \delta$ and let $\tilde c$ be a degree of freedom.
	First note that choosing $\sigma_1 = (2C_2)^{-1}$, then choose $c_3 = \tilde c c_4$ and $c_2 = [\bar c_3]^{\tilde c}$, then for a large enough $\tilde c = \tilde c(p,n,M)$
	\begin{align} \label{forward ineq choice 1}
		2 \bar c_3 (2C_2)^{p-2} \tilde c^{1-p} \bar c_3^{(p-2)\tilde c / \delta} \left( \frac{r}{\varrho} \right)^{(p-2)\bar c_4 \tilde c / \delta} <  c_2^{(p-2)/\delta} \left( \frac{r}{\varrho} \right)^{(p-2)c_3/\delta}\,.
	\end{align}
	Second we see that choosing $\sigma_2 = c_5^{-1/\delta}$ for large enough $c_5 = c_5(p,n,M)$ the following holds
	\begin{equation} \label{forward ineq choice 2}
		 \frac12 \bar c_1 \sigma_2^{2-p} \tilde c^{1-p} \left( \frac{r}{\varrho} \right)^{(p-2)\bar c_4 \tilde c / \delta} > c_2^{(p-2)/\delta} \left( \frac{r}{\varrho} \right)^{(p-2)c_3/\delta}\,.
	\end{equation}
	With \Cref{forward ineq choice 1,forward ineq choice 2,forward fundamental inequality} at hand we see that there is a choice of $\sigma \in [\sigma_2,\sigma_1]$ such that
	\begin{align*}
		\mathbf{T} = \delta^{p-1} c_2^{(p-2)/\delta} \left( \frac{r}{\varrho} \right)^{(p-2)c_3/\delta}\,,
	\end{align*}
	and thus we have proved $t_{\hat k+1} = t_0 + \mathbf{T} \Lambda^{2-p} r^p = t_0 + \tau$.
	This together with \Cref{eq:NTA FCM final bound} finishes the proof by taking suitably large $c_1$ in the statement.
\end{proof}

For $p$-parabolic functions we have the following pointwise version of \Cref{NTA FChainMeas}.

\begin{theorem}
	\label{NTA FChain} Let $\Omega\subset\R^n$ be an NTA-domain with constants $M$ and $r_0$, let $x_0 \in \partial \Omega$, $T>0$ and let $0 < r < r_0$. Let $x,y$ be two points in $\Omega \cap B_r(x_0)$ such that
	\begin{equation*}
		\varrho := d(x,\partial \Omega) \leq r \qquad \mbox{and} \qquad d(y,\partial \Omega) \geq \frac{r}{4}\,.
	\end{equation*}
	Assume that $u$ is a non-negative $p$-parabolic function in $\Omega_T$, and assume that $u(x,t_0)$ is positive. Let $\delta \in (0,1]$. Then there exist constants $c_i \equiv c_i(M,p,n)$, $i\in\{1,2,3\}$, such that if
	\begin{equation*}
		t_0 - (c_h/u(x,t_0))^{p-2} (\delta \varrho)^p > 0\,, \qquad t_0 + \tau < T\,,
	\end{equation*}
	with
	\begin{equation*}
		\tau := \delta^{p-1} \left( c_2^{-1/\delta} \left( \frac{r}{\varrho} \right)^{- c_3/\delta} u(x,t_0) \right)^{2-p} r^p \,,
	\end{equation*}
	then
	\begin{equation*}
		u(x,t_0) \leq c_1^{1/\delta} \left( \frac{r}{\varrho} \right)^{c_3/\delta} \inf_{z \in B_{r/16}(y)} u(z ,t_0 + \tau)\,.
	\end{equation*}
	Furthermore, constants $c_i$, $i \in \{1,2,3\}$, are stable as $p \to 2^+$.
\end{theorem}

\begin{proof}
	Applying \Cref{Har} once, we see that the theorem follows from \Cref{NTA FChainMeas}.
\end{proof}


\subsection{Backward Harnack chains}
The philosophy of the forward Harnack chains in \Cref{ss.fchain} is that the data at the starting point will start to diffuse according to the intrinsic Harnack inequality. The finite speed of diffusion forces the waiting time to blow up if we wish to spread our information in an infinite chain. In the backward Harnack chains that we develop in this section the philosophy is reversed. Instead of looking to the future we look to the past. This means that if the value of the solution at a point $(y,s)$ is, say, one, then we ask the question: how large can the values in the past be without violating the fact that the solution is one at $(y,s)$.

We start with the weak version of the backward Harnack chains, valid for weak supersolutions.

\begin{theorem}
	\label{NTA BChainMeas} Let $\Omega\subset\R^n$ be an NTA-domain with with constants $M$ and $r_0$, let $x_0 \in \partial \Omega$, $T>0$, and let $0 < r < r_0$. Let $x,y$ be two points in $\Omega \cap B_r(x_0)$ such that
	\begin{equation*}
		\varrho := d(x,\partial \Omega) \leq r \qquad \mbox{and} \qquad d(y,\partial \Omega) \geq \frac{r}{4}\,.
	\end{equation*}
	Assume that $u$ is a non-negative continuous weak supersolution in $\Omega_T$, and assume that $u( y , s)$ is positive. Let $\delta \in (0,1]$. Then there exist positive constants $C_i \equiv C_i(p,n)$ and $c_i \equiv c_i(p,n,M)$, $i\in\{4,5\}$, such that if $s \in (\tau,T)$ and
	\begin{equation} \nonumber
		t \in [s - \tau\,, s - \delta^{p-1}\tau]\,,
	\end{equation}
	with
	\begin{equation*}
		\tau := C_4 \left[C_5 u(y,s) \right]^{2-p} r^p
	\end{equation*}
	then
	\begin{equation} \label{NTA BChainMeas est}
		\fint_{B_{\varrho/4}(x)} u(z,t)\de z \leq c_4^{1/\delta} \left( \frac{r}{\varrho} \right)^{c_5/\delta} u(y,s)\,.
	\end{equation}
	Furthermore, constants $c_i,C_i$, $i \in \{4,5\}$, are stable as $p \to 2^+$.
\end{theorem}
\begin{remark}
	If we assume that $u \in C([0,T);L^2(\Omega))$ then we can replace $s \in (\tau,T)$ with $s \in [\tau,T)$. That is, the chain can be taken all the way to the initial time, i.e. $t = 0$.
\end{remark}
\begin{proof}
	After scaling, we may assume that $u(y,s) = 1$. Assume now the contrary to \Cref{NTA BChainMeas est}, i.e.
	\begin{equation}
		\label{eq:bh_counterMeas} \fint_{B_{\varrho/4}(x)} u(z,t)\de z > H \left( \frac{r}{\varrho} \right)^{c_5/\delta}
	\end{equation}
	for constants $c_5,H$ to be fixed. Then~\Cref{NTA FChainMeas} implies that with
	\begin{align*}
		\tilde \tau := & \tilde \delta^{p-1} \left( c_2^{-1/\tilde \delta} \left( \frac{r}{\varrho} \right)^{- c_3/\tilde \delta} \fint_{B_{\varrho/4}(x)} u(z,t)\de z \right)^{2-p} r^p
	\end{align*}
	we get
	\begin{equation}
		\label{eq:bh_harnackMeas} \fint_{B_{\varrho/4}(x)} u(z,t)\de z \leq c_1^{1/\tilde \delta} \left( \frac{r}{\varrho} \right)^{c_3/\tilde \delta} \inf_{z \in B_{r/16}(y)} u(z,t + \tilde \tau)
	\end{equation}
	with constants $c_i \equiv c_i(p,n,M)$, $i \in \{1,2,3\}$ and $\tilde \delta \in (0,\delta]$. Now we have an upper bound for $\tilde \tau$ by means of~\cref{eq:bh_counterMeas} as follows:
	\begin{align*}
		\tilde \tau \leq & \tilde \delta^{p-1} \left( c_2^{-1/\tilde \delta} \left( \frac{r}{\varrho} \right)^{- c_3/\tilde \delta} H \left( \frac{r}{\varrho} \right)^{c_5/\delta} \right)^{2-p} r^p \\
		= & \tilde \delta^{p-1} \left( c_2^{-1/\tilde \delta} H \right)^{2-p} r^p \\
		\leq & \delta^{p-1} \tau\,,
	\end{align*}
	provided that
	\begin{equation}
		\label{eq:bh_choice} H \geq C_5 c_2^{1/\tilde \delta} \,, \quad \tilde \delta := \delta \min\{1, C_4\}^{1/(p-1)}\,, \quad c_5 := \frac{c_3}{\min\{1, C_4\}^{1/(p-1)}}\,.
	\end{equation}
	Therefore we have that
	\begin{equation*}
		t + \tilde \tau \leq s\,.
	\end{equation*}
	Observe that both $C_4$ and $C_5$ are still to be fixed. Thus we need to carry the information from the time $t + \tilde \tau$ up to $s$. To this end, connecting~\cref{eq:bh_counterMeas,eq:bh_harnackMeas} with the choices in \Cref{eq:bh_choice} leads to
	\begin{equation}
		\label{eq:bh_step1Meas} H c_1^{-1/\tilde \delta} < \inf_{z \in B_{r/16}(y)} u(z,t + \tilde \tau) \,.
	\end{equation}
	Truncate $u$ as
	\begin{equation*}
		\tilde u = \min \left(4 C_2 , u \right) \,,
	\end{equation*}
	and take
	\begin{equation} \label{eq:bh_choice2}
		H := c_4^{1/\delta} \,, \qquad c_4 := \max\{ 4 C_2 c_1, C_5 c_2\} ^{1/\min\{1, C_4\}^{1/(p-1)}}\,,
	\end{equation}
	where $C_2$ is as in~\Cref{TuomoWH}. Then $\tilde u$ is a continuous weak supersolution, and we have by~\cref{eq:bh_choice,eq:bh_choice2,eq:bh_step1Meas} that
	\begin{equation*}
		\fint_{B_{\tilde r}(y) } \tilde u(z , t+ \tilde \tau) \de z = 4 C_2 \,, \qquad \tilde r \in (0,r/16]\,.
	\end{equation*}
	Applying thus the forward in time weak Harnack estimate in~\Cref{TuomoWH} gives
	\begin{equation*}
		4C_2 \leq 2C_2 \inf_{z \in B_{2 \tilde r}(y)} \tilde u \left(z, t + \tilde \tau+ C_1 \left( 4 C_2 \right)^{2-p} \tilde r^p\right) \,,
	\end{equation*}
	provided that
	\begin{equation*}
		t + \tilde \tau+ C_1 \left( 4 C_2 \right)^{2-p} \tilde r^p < T\,.
	\end{equation*}
	Choosing $C_4 = 16^{-p} C_1$ and $C_5 = 4C_2$, we always find $\tilde r \leq r/16$ such that
	\begin{equation*}
		t + \tilde \tau+ C_1 \left( 4 C_2 \right)^{2-p} \tilde r^p = s\,,
	\end{equation*}
	and hence
	\begin{equation*}
		2 \leq \inf_{z \in B_{2 \tilde r}(y)} \tilde u \left(z, s \right) \,.
	\end{equation*}
	This gives a contradiction since we assumed that $u(y,s)=1$, and thus the proof is complete.
\end{proof}

The following theorem is the corresponding result for weak solutions, where we use pointwise information in the past instead of information in mean. The main difference this imposes on the assumptions in \Cref{NTA BChainMeas} is that we have to require the solution to have lived for a certain amount of time, which is precisely the price we have to pay if we wish to control pointwise values in the past.

\begin{theorem}
	\label{NTA BChain} Let $\Omega\subset\R^n$ be an NTA-domain with with constants $M$ and $r_0$, let $x_0 \in \partial \Omega$, $T>0$, and let $0 < r < r_0$. Let $x,y$ be two points in $\Omega \cap B_r(x_0)$ such that
	\begin{equation*}
		\varrho := d(x,\partial \Omega) \leq r \qquad \mbox{and} \qquad d(y,\partial \Omega) \geq \frac{r}{4}\,.
	\end{equation*}
	Assume that $u$ is a non-negative $p$-parabolic function in $\Omega_T$, and assume that $u( y , s)$ is positive. Let $\delta \in (0,1]$. Then there exist positive constants $C_i \equiv C_i(p,n)$ and $c_i \equiv c_i(p,n,M)$, $i\in\{4,5\}$, such that if $s<T$ and
	\begin{equation}
		\label{eq:NTA BChain t cond} \max\left\{ \left(\frac{c_4^{1/\delta}}{c_h} \left( \frac{r}{\varrho} \right)^{c_5/\delta} u(y,s) \right)^{2-p } (\delta \varrho)^p , s - \tau \right\}\, \leq\, t\, \leq\, s - \delta^{p-1} \tau\,,
	\end{equation}
	with
	\begin{equation*}
		\tau := C_4 \left[C_5 u(y,s) \right]^{2-p} r^p
	\end{equation*}
	then
	\begin{equation*}
		u(x,t) \leq c_4^{1/\delta} \left( \frac{r}{\varrho} \right)^{c_5/\delta} u(y,s)\,.
	\end{equation*}
	Furthermore, constants $c_i,C_i$, $i \in \{4,5\}$, are stable as $p \to 2^+$.
\end{theorem}

\begin{proof}
	To prove the lemma we follow the same outline as the proof of \Cref{NTA BChainMeas}, but instead of assuming the contrary assumption \Cref{eq:bh_counterMeas} we assume the contrary assumption that
	\begin{equation*}
		u(x,t) > H \left( \frac{r}{\varrho} \right)^{c_5/\delta}\,,
	\end{equation*}
	for some constants $c_5,H$ to be fixed. Applying \Cref{NTA FChain} instead of \Cref{NTA FChainMeas} we get
	\begin{equation*}
		u(x,t) \leq c_1^{1/\tilde \delta} \left( \frac{r}{\varrho} \right)^{c_3/\tilde \delta} \inf_{z \in B_{r/16}(y)} u(z,t + \tilde \tau) \,.
	\end{equation*}
	Note that it is the usage of \Cref{NTA FChain} which requires the assumption \Cref{eq:NTA BChain t cond}. The proof now follows repeating the remaining part of the proof of \Cref{NTA BChainMeas} essentially verbatim.
\end{proof}


\section{Carleson estimate} \label{s.Carleson}

In this section we  prove, using the improved Harnack chain estimate in \Cref{NTA BChain}, a flexible Carleson estimate valid in cylindrical NTA-domains. Versions of  the Carleson estimate was originally proved, for equations of $p$-parabolic type in~\cite{AGS} in Lipschitz cylinders, and in \cite{A} in certain time-dependent space-time domains. We begin with a standard oscillation decay lemma valid for weak subsolutions, see for example \cite{DB}. 

\begin{lemma}
	\label{lemoscdecay} Let $\Omega \subset \R^n$ be an NTA-domain with constants $M$ and $r_0$. Let $u$ be a non-negative, continuous weak subsolution in $ \Omega_T$. Let $(x_0,t_0) \in S_T$, $0 < r < r_0$,
	\begin{equation*}
		Q_r^\lambda(x_0,t_0) := (B_r(x_0) \cap \Omega) \times (t_0 - \lambda^{2-p} r^p, t_0) \subset \Omega_T\,,
	\end{equation*}
	and assume that $u$ vanishes continuously on $S_T \cap Q_r^\lambda(x_0,t_0)$ and that
	\begin{equation*}
		 \sup_{Q_r^\lambda(x_0,t_0)} u \leq \lambda\,.
	\end{equation*}
	Then there is a constant $\sigma$ depending only on $p,n,M$ such that $Q_{\sigma r}^{\lambda/2}(x_0,t_0) \subset Q_r^\lambda(x_0,t_0) $ and
	\begin{equation*}
		\label{eq:osc red 1} \sup_{Q_{ \sigma r}^{ \lambda /2 }(x_0,t_0) } u \leq \frac{\lambda}{2}\,.
	\end{equation*}
	In particular, we have that
	\begin{equation*}
		\label{eq:osc red 2} \sup_{Q_{ \sigma^j r}^{ 2^{-j} \lambda }(x_0,t_0) } u \leq 2^{-j}\lambda\,.
	\end{equation*}
	for any $j \in \mathbb{N}$.
\end{lemma}

The following theorem is usually referred to as a Carleson estimate. We want to point out that, compared to~\cite{AGS}, not only does it hold for cylindrical NTA-domains, but also the formulation is more flexible for applications. In particular, we are able to adjust the waiting time, the height of the cylinder, and the distance to the initial boundary. All these parameters influence the constant in the inequality and a Gaussian type behavior is proved to be present.

\begin{theorem}
	\label{carleson:thm1} Let $\Omega \subset \R^n$ be an NTA-domain with constants $M$ and $r_0$. Let $u$ be a non-negative, weak solution in $\Omega_T$. Let $(x,t) \in S_T$ and $0 < r < r_0$. Assume that $u(a_r(x),t) >0$ and let
	\begin{equation*}
		\tau = \frac{C_4}{4} \left[C_5 u(a_r(x),t) \right]^{2-p} r^p \,,
	\end{equation*}
	where $C_4$ and $C_5$, both depending on $p,n$, are as in~\Cref{NTA BChain}. Assume that $t > (\delta_1^{p-1}+\delta_2^{p-1} + 2\delta_3^{p-1}) \tau$ for $0< \delta_1\leq \delta_3 \leq 1$, $\delta_2 \in (0,1)$ and that for a given $\lambda \geq 0$, the function $(u - \lambda)_+$ vanishes continuously on $S_T \cap B_r(x) \times (t-(\delta_1^{p-1} + \delta_2^{p-1}+\delta_3^{p-1}) \tau, t-\delta_1^{p-1} \tau)$ from $\Omega_T$. Then there exist constants $c_i \equiv c_i(M,p,n)$, $i\in \{6,7\}$, such that
	\begin{equation*}
		\sup_{Q} u \leq \left( c_6 / \delta_3 \right)^{c_7/ \delta_1} u(a_r(x),t) + \lambda\,,
	\end{equation*}
	where $Q := B_r(x) \times ( t - (\delta_1^{p-2}+\delta_2^{p-1})\tau, t - \delta_1^{p-1} \tau)$. Furthermore, constants $c_i$, $i \in \{6,7\}$, are stable as $p \to 2^+$
\end{theorem}
\begin{remark}
	Note that in the case $p=2$, \Cref{carleson:thm1} for $\lambda = 0$ is equivalent to the estimate above by linearity. However, for $p > 2$ this result extends the ones in \cite{A,AGS} if $\lambda > 0$.
\end{remark}
\begin{proof}
	By scaling the function $u$ we can assume that $u(a_r(x),t) = 1$, and replacing $\lambda$ with its scaled version. Consider the boxes
	\begin{equation*}
		\label{eq:carlbox1} \widetilde Q_\nu := \big(\Omega \cap B_{r}(x)\big) \times (t - (\delta_1^{p-1}+\delta_2^{p-1} + \delta_3^{p-1}) \tau, t - \nu^{p-1} \tau)\,, \qquad \nu \in \{\delta_1, \delta_3\}\,,
	\end{equation*}
	and denote also
	\begin{equation*}
		Q_r^\theta(x,t) := (\Omega \cap B_r(x)) \times (t - \theta^{2-p} r^p,t)\,, \qquad \theta>0\,.
	\end{equation*}
	Observe that with the choice of $\tau$, we have by~\Cref{NTA BChain}, for any $(y,s) \in \widetilde Q_\nu$ with $d(y,\partial \Omega) \leq r$ and $\nu \in \{\delta_1, \delta_3\}$, that
	\begin{equation}
		\label{eq:carleson_bh1} u(y,s) \leq c_8^{1/\nu} \left( \frac{r}{d(y,\partial \Omega)} \right)^{c_9/\nu}
	\end{equation}
	holds with $c_8,c_9$ depending only on $p,n,M$ (apply~\Cref{NTA BChain} with the choice $\delta := \nu \min\{1,C_4/2\}$ in order to guarantee the conditions in~\cref{eq:NTA BChain t cond}).
	
	We proceed by induction via a contradiction assumption. Assume that $P_0 = (x_0,t_0) \in Q $, where $Q$ is as in the statement, is such that $u(P_0) > H + \lambda$ for some large $H$ to be fixed. Assume then that we find inductively points $P_j = (x_j,t_j) \in \widetilde Q_{\hat \delta_j}$, where $\hat \delta_j = \delta_3$ if $t_j \leq t_0 - \delta_3^{p-1} \tau/2$ and $\hat \delta_j = \delta_1$ if $t_j > t_0 - \delta_3^{p-1} \tau/2$ for any $j \in \mathbb{N}$. Set $r_j := d(x_j,\partial \Omega)$ and let $x_j' \in \partial \Omega$ be such that $r_j = |x_j-x_j'|$. Assume inductively that
	\begin{equation}
		\label{eq:indass}
		\begin{cases}
			u(P_j) \,>\, 2^j H + \lambda \quad \text{and} \\
			t - (\delta_1^{p-1}+\delta_2^{p-1}+\delta_3^{p-1}) \tau\, <\, t_{j} \,\leq\, t_{j-1} \,\leq\, t_{j} + (2^j H)^{2-p} (r_{j-1} /\sigma)^p
		\end{cases}
	\end{equation}
	holds for all $j \in \{1,\ldots,k\}$, where $\sigma \equiv \sigma(p,n,M) \in (0,1)$ is as in~\Cref{lemoscdecay}. We then want to show that for large enough $H$ this continues to hold for $j=k+1$ as well.
	
	To show the induction step, observe that~\cref{eq:carleson_bh1} and the induction assumption~\Cref{eq:indass} imply that
	\begin{equation*}
		2^jH + \lambda < u(P_j) \leq c_8^{1/\hat \delta_j } \left( \frac{r}{r_j} \right)^{c_9/\hat \delta_j} \qquad \Longrightarrow \qquad r_j \leq \left(2^j H c_8^{- 1/\hat \delta_j } \right)^{- \hat \delta_j / c_9} \,.
	\end{equation*}
	Fixing
	\begin{equation*}
		H := \left( \frac{ 4 c_8^{p/c_9} c_9 }{p \sigma^p \log 2 } \frac{\tau}{r^p \delta_3^p} \right)^{c_9/[p \delta_1]} =: \left( c_6 / \delta_3 \right)^{c_7/ \delta_1}\,,
	\end{equation*}
	we have after simple manipulations that
	\begin{equation}
		\label{eq:r_j cond 111} r_j^p \leq \sigma^p \frac{\delta_3^{p-1} \tau}{4} 2^{- j \hat \delta_j p / c_9} \frac{\delta_3 p \log 2 }{c_9 } \leq \sigma^p \frac{\delta_3^{p-1} \tau}{4} \frac{2^{- j \hat \delta_j p/ c_9}}{\sum_{j=0}^\infty 2^{-j \delta_3 p / c_9} }\,.
	\end{equation}
	In particular,
	\begin{equation*}
		(2^{k+1} H)^{2-p} (r_k/\sigma)^p \leq (r_k/\sigma)^p \leq \frac{\delta_3^{p-1} \tau}{4}\,.
	\end{equation*}
	We hence set $K_1 := Q_{r_k}^{2^k H}(x_k',t_k)$ and $K_2 := Q_{r_k/\sigma}^{2^{k+1}H}(x_k',t_k)$, and we deduce, by the induction assumption and the estimate in previous display, that $K_2 \subset \Omega_T$. Now, if
	\begin{equation*}
		\sup_{K_2} (u-\lambda)_+ \leq 2^{k+1}H\,,
	\end{equation*}
	then, using that  $(u-\lambda)_+$ is a weak subsolution,~\Cref{lemoscdecay} would imply that
	\begin{equation*}
		\sup_{K_1} u \leq 2^k H + \lambda\,,
	\end{equation*}
	which is a contradiction since $P_k \in \overline{K}_1$. Thus there is $P_{k+1} \in K_2$ such that
	\begin{equation*}
		\label{eq:carlfinal1} u(P_{k+1}) > 2^{k+1}H + \lambda\,.
	\end{equation*}
	By the definition of $K_2$ and $P_{k+1}$, we must have that
	\begin{equation*}
		t_{k+1} \leq t_{k} \leq t_{k+1} + (2^{k+1} H)^{2-p} (r_k/\sigma)^p \,.
	\end{equation*}
	Therefore we are left to show that $t_{k+1} \geq t - (\delta_1^{p-1}+ \delta_2^{p-1}+\delta_3^{p-1}) \tau$ in order to prove the induction step.
	
	To this end, let now $\hat k \leq k+1$ be the largest integer such that $t_0-t_{\hat k} \leq \delta_3^{p-1} \tau/2 $. We may without loss of generality assume that $\hat k < k+1$, since otherwise $t_{k+1} \geq t-(\delta_1^{p-1}+\delta_2^{p-1}+\delta_3^{p-1})\tau$, because $t_0 > t - (\delta_1^{p-1} + \delta_2^{p-1}) \tau$. Now~\cref{eq:r_j cond 111,eq:indass}, together with the fact that $\hat \delta_j = \delta_3$ for $j > \hat k$, give
	\begin{align}
		\label{eq:tsum} \nonumber t_0 -t_{k+1} =& (t_0 - t_{\hat k}) + (t_{\hat k}-t_{k+1}) \\
		\leq & \nonumber \frac{\delta_3^{p-1}\tau }{2} + \sum_{j=\hat k}^{k} (t_{j} - t_{j+1}) \\
		\leq & \nonumber \frac{\delta_3^{p-1} \tau }{2} + (2^{\hat k+1} H)^{2-p} (r_{\hat k}/\sigma)^p + \sum_{j=\hat k+1}^{k} (2^{j+1} H)^{2-p} (r_{j} /\sigma)^p \\
		\leq & \nonumber \frac{\delta_3^{p-1} \tau }{2} + \frac{\delta_3^{p-1} \tau }{4} + \frac{\delta_3^{p-1} \tau }{4} \left[ \sum_{j=0}^\infty 2^{-j \delta_3 p/ c_9} \right]^{-1} \sum_{j=\hat k+1}^{k} 2^{-j \delta_3 p/ c_9} \\
		< & \nonumber \delta_3^{p-1} \tau \,.
	\end{align}
	Therefore, since $t_0> t - (\delta_1^{p-1}+\delta_2^{p-1})\tau$, we have that
	\begin{equation*}
		 t - t_{k+1} = t - t_0 + t_0 -t_k < (\delta_1^{p-1}+\delta_2^{p-1}+\delta_3^{p-1}) \tau \,,
	\end{equation*}
	which was to be proven. Hence we have concluded the proof of the induction step. As a consequence, we have constructed a sequence of points $P_j = (x_j,t_j) \in \widetilde Q_{\delta_1}$,  such that $d(x_j,\partial \Omega) \to 0$ and $u(P_j) \to \infty$ as $j \to \infty$. This violates the assumed continuity of $(u-\lambda)_+$ in the neighborhood of $S_T \cap \overline{ \widetilde Q_{\delta_1}}$, giving the contradiction. Hence,  
	\begin{equation*}
		\sup_Q u \leq H + \lambda\,,
	\end{equation*}
	completing the proof of the theorem.
\end{proof}


\section{Estimating the boundary type Riesz measure}  \label{s.measure}
In this section we establish, in NTA-cylinders, upper and lower bounds for the measure $\mu$ defined in~\Cref{1.1+}. 


\subsection{Upper estimate on $\mu$}
We will first provide an upper bound on the measure. The proof relies on the Carleson estimate in \Cref{carleson:thm1} and the following standard Caccioppoli type estimate, see~\cite{DB}.
\begin{lemma}
	\label{lem:energy} Let $u$ be a non-negative weak subsolution in $\Omega_T$, and $\phi \in C_0^\infty(\Omega \times (t_1,T))$ with $\phi \geq 0$. Then
	\begin{equation*}
		\int_{t_1}^{t_2} \int_{\Omega} |\grad u|^p \phi^p \dx\dt \leq C \left ( \int_{t_1}^{t_2} \int_{\Omega} u^p |\grad \phi|^p \dx \dt + \int_{t_1}^{t_2}\int_{\Omega} u^2 (\phi_t)_+\phi^{p-1}\dx\dt \right )
	\end{equation*}
	for $C = C(p,n)$.
\end{lemma}
\begin{theorem}
	\label{MuUpperBound} Let $\Omega \subset \R^n$ be an NTA-domain with constants $M$ and $r_0$. Let $0<r\leq r_0$ and let $u$ be a weak non-negative solution in $\Omega_T$. Fix a point $x_0 \in \partial \Omega$ and define
	\begin{equation*}
		\tau = \frac{C_4}{16} \left[C_5 u(a_r(x_0),t_0) \right]^{2-p} r^p \,,
	\end{equation*}
	where $C_4$ and $C_5$, both depending on $p,n$, are as in~\Cref{NTA BChain}. Let $0<\delta \leq \tilde \delta \leq 1$ and assume that $t_0 > 5 \tilde \delta^{p-1} \tau $ and that $u$ vanishes continuously on $S_T \cap \bigl (B_r(x) \times (t_0-4\tilde \delta^{p-1} \tau, t_0-\delta^{p-1} \tau)\bigr )$ from $\Omega_T$. Then there is a constant $C \equiv C(p,n)$ and $c_8 \equiv c_8(p,n,M)$ such that
	\begin{equation*}
		\frac{\mu(Q)}{r^n} \leq C \left( c_6 / \tilde \delta \right)^{c_8/ \delta} u(a_r(x_0),t_0)\,,
	\end{equation*}
	where  $\mu$ is the measure from~\cref{1.1+},
	\begin{equation*}
		Q := B_{r/2}(x_0) \times (t_0 - 2 \tilde \delta^{p-1} \tau , t_0 - \delta^{p-1} \tau) \,,
	\end{equation*}
	and $c_6$ is from \Cref{carleson:thm1}. Furthermore, constants $C$, $c_8$, are stable as $p \to 2^+$.
\end{theorem}
\begin{proof}
	After scaling, we may assume that $u(a_r(x_0),t_0)= 1$. Let
	\begin{equation*}
		\hat Q = B_r(x_0) \times (t_0 - 3 \tilde \delta^{p-1} \tau , t_0 - \delta^{p-1} \tau)\,,
	\end{equation*}
	and observe, by our assumptions, that \Cref{carleson:thm1} implies
	\begin{equation}
		\label{eq:uprbndruprest} \sup_{\hat Q} u \leq \left( c_6 / \tilde \delta \right)^{c_7/ \delta} =: \Lambda\,.
	\end{equation}
As in the construction of the measure $\mu$ in \Cref{1.1+}, we see that extending $u$ to the entire cylinder $\hat Q$ as zero, we obtain a weak subsolution in $\hat Q$.	
	Take a cut-off function $\phi \in C^\infty(\hat Q)$ vanishing on $\partial_p \hat Q$ such that $0 \leq \phi \leq 1$, $\phi$ is one on $Q$, and $|\grad \phi| < 4/r$ and $(\phi_t)_+ < 4/[\tilde \delta^{p-1} \tau]$. Then by~\cref{1.1+}, the definition of~$\phi$ and H{\"o}lder's inequality we get
	\begin{align*}
		\int_{\hat Q} \phi^p \de \mu &\leq \int_{\hat Q} |\grad u|^{p-1} |\grad \phi| \phi^{p-1} \dx \dt + \int_{\hat Q} u (\phi_t)_+ \phi^{p-1}\dx \dt \\
		&\leq \frac{4}{r} \int_{\hat Q} |\grad u|^{p-1} \phi^{p-1} \dx \dt + \int_{\hat Q} u (\phi_t)_+ \phi^{p-1} \dx\dt \\
		&\leq \frac{4}{r} |\hat Q|^{1/p} \left ( \int_{\hat Q} |\grad u|^p \phi^p \dx \dt \right )^{\frac{p-1}{p}} + \int_{\hat Q} u (\phi_t)_+ \phi^{p-1} \dx \dt \,.
	\end{align*}
	Now using \Cref{lem:energy} and~\cref{eq:uprbndruprest} we see that
	\begin{align*}
		\mu(Q) &\leq C \frac{|\hat Q|^{1/p}}{r} \left ( \int_{\hat Q} u^p |\grad \phi|^p + u^2 (\phi_t)_+ \phi^{p-1} \dx \dt \right )^{\frac{p-1}{p}} + \int_{\hat Q} u (\phi_t)_+ \phi^{p-1} \dx \dt \,, \\
		&\leq C \frac{|\hat Q|^{1/p}}{r} \left (|\hat Q| \left (\frac{\Lambda^p}{r^{p}} + \frac{\Lambda^2}{\tilde \delta^{p-1} \tau } \right ) \right )^{\frac{p-1}{p}} + C |\hat Q| \frac{\Lambda}{\tilde \delta^{p-1} \tau } \\
		&\leq C \frac{|\hat Q|}{\tilde \delta^{p-1} r^p} \Lambda^{p-1}\,.
	\end{align*}
	After scaling back, this can be rewritten in the homogeneous form
	\begin{equation*}
		\frac{\mu(Q_r)}{r^n} \leq C \left( c_6 / \tilde \delta \right)^{(p-1)c_7/ \delta} u(a_r(x_0),t_0)\,,
	\end{equation*}
	completing the proof with $c_8 = (p-1)c_7$.
\end{proof}


\subsection{Lower estimate on $\mu$}

We next prove the lower bound for the measure~$\mu$.

\begin{theorem}
	\label{MuLowerBound} Let $\Omega \subset \R^n$ be an NTA-domain with constants $M$ and $r_0$, and let $u$ be a weak non-negative solution in $\Omega_T$. Fix a point $(x_0,t_0) \in \partial \Omega \times (0,T]$, and denote $A_r^- = (a_{r/2}(x_0),t_0)$ for $0 < r < r_0$.  There exists $C, \tau_0 , \tau_1$, all depending only on $p,n,M$, such that if
	\begin{equation*}
		\left(t_0 - \tau_0 u(A_r^-)^{2-p} r^p , t_0+ (\tau_0+\tau_1) u(A_r^-)^{2-p}r^p \right) \subset (0,T)\,,
	\end{equation*}
	and if $u$ vanishes continuously on $S_T \cap \left( B_r(x) \times (t_0, t_0+ (\tau_0+\tau_1) u(A_r^-)^{2-p}r^p) \right)$ from $\Omega_T$, then
	\begin{equation*}
		u(A_r^-) \leq C \frac{\mu(Q)}{r^n}\,,
	\end{equation*}
	where $\mu$ is the measure from~\cref{1.1+} and
	\begin{equation*}
		Q := B_{r}(x_0) \times \left(t_0 + \tau_0 u(A_r^-)^{2-p}r^p, t_0 + (\tau_0+\tau_1) u(A_r^-)^{2-p}r^p \right) \,.
	\end{equation*}
	Furthermore, constants $C, \tau_0, \tau_1$, are stable as $p \to 2^+$.
\end{theorem}

 To prove \Cref{MuLowerBound} we first consider the model problem in \Cref{MuLowerSimple}, and we prove that the measure associated  to this model problem is bounded from below by a constant. Returning to \Cref{MuLowerBound}, we then apply the intrinsic Harnack inequality to obtain a lower bound on the function such that by the comparison principle the solution $v$ in \Cref{MuLowerSimple} is below our solution $u$. The result then follows by the fact that the corresponding measures are ordered according to \Cref{meas order}, a fact easily realized if the domain is smooth, as in this case the measure is just the modulus of the gradient to the power $p-1$.

\begin{lemma}
	\label{MuLowerSimple} Let $\Omega \subset \R^n$ be an NTA-domain with constants $M$ and $r_0=2$. There exists constants $C,T_M$, both depending on $p,n,M$, such that if $v$ is a continuous solution to the problem
	\begin{equation*}
		\label{eq:vDirichlet}
		\begin{cases}
			v_t - \Delta_p v = 0 &\text{ in } (\Omega \cap B_2(0)) \times (0,T_0) \\
			v = 0 &\text{ on } \partial (\Omega \cap B_2(0)) \times [0,T_0) \\
			v = \chi_{B_{1/(4M)}(a_1(0))} & \text{ on } (\Omega \cap B_2(0)) \times \{0 \}\,,
		\end{cases}
	\end{equation*}
	then
	\begin{equation*}
		\mu_v\big(B_2(0) \times (0,T_0)\big) \geq 1/C\,.
	\end{equation*}
	Furthermore, constants $C, T_M$, are stable as $p \to 2^+$
\end{lemma}
\begin{proof}
	To begin with, extend $v$ as zero to the rest of $Q = B_2(0) \times (0,T_0)$, i.e., set $v \equiv 0$ in $(B_2(0) \setminus \Omega) \times (0,T_0)$, and let $\mu_v$ be the associated measure as in \Cref{1.1+}. Let then $h$ be the solution to the problem
	\begin{equation*}
		\begin{cases}
			h_t - \Delta_p h = 0 &\text{ in } Q \\
			h = w &\text{ on } \partial_p Q \,.
		\end{cases}
	\end{equation*}
	We observe that the supremum of $h$ and $v$, which is one, is attained at the bottom of the cylinder. Let us now recall the decay estimate in \Cref{lemoscdecay}, which implies that for $Q_r^\lambda(0,t_0) := (\Omega \cap B_r(0)) \times (t_0 - \lambda^{2-p} r^p , t_0)$ we have
	\begin{equation}
		\label{eq:v decay} \sup_{Q_{\sigma^j}^{2^{-j} }(0,t_0) } v \leq 2^{-j}
	\end{equation}
	for $j \in \mathbb{N}$ provided that $t_0 \in [1,T_0]$ and $T_0 > 1$. On the other hand, \Cref{FChainMeas} gives us that
	\begin{equation*}
		1 = \fint_{B_{(4M)^{-1}}(a_1(0))} h(x,0) \dx \leq C \inf_{z \in B_{(2M)^{-1}}(0)} u(z, \tau )
	\end{equation*}
	with $\tau$ and $C$ depending on $p,n,M$. We then apply \Cref{TuomoWH} in order to get
	\begin{equation}
		\label{eq:inf h} 1 \leq \hat C \inf_{\hat Q} h\,, \qquad \hat Q := B_{1/M}(0) \times (T_0/2,T_0)
	\end{equation}
	by properly choosing $T_0$ by means of $\tau$ and $C$ to be larger than $2$. We then choose large enough $j^\filledstar \in \mathbb{N}$ so that
	\begin{equation*}
		2^{-j^\filledstar} \leq \frac{1}{2 \hat C} \qquad \mbox{and} \qquad \sigma^j \leq 1/M \,.
	\end{equation*}
	Then, sliding $t_0$ along $(1,T_0]$ in~\cref{eq:v decay}, we obtain by combining~\cref{eq:v decay,eq:inf h} that there is $r_1 = r_1(p,n,M)$ such that
	\begin{equation*}
		\inf_{\tilde Q} (h-v) \geq \frac{1}{2 \hat C} =: \epsilon \,, \qquad \tilde Q := B_{r_1}(0) \times (T_0/2,T_0)\,,
	\end{equation*}
	where $\epsilon \equiv \epsilon(p,n,M)$.
	
	Let us now define $\phi = \min \{ h-v, \epsilon \}$, which is vanishing on $\partial_p Q$ and is $\epsilon$ on $\tilde Q$. Then from the weak formulation of $h$ and $v$ we get that
	\begin{align*}
		\int_{Q \cap \{t<\tau \}} &(h-v)_t \phi + \big ( |\grad h|^{p-2} \grad h - |\grad v|^{p-2}\grad v \big ) \cdot \grad \phi \dx \dt = \int_{Q\cap \{t<\tau \}} \phi \de \mu_w \,.
	\end{align*}
	For the time term, we integrate to obtain
	\begin{align*}
		\int_{Q\cap \{t<\tau \}} (h-v)_t \phi \dx \dt & = \int_{Q\cap \{t<\tau \}} \partial_t \left( \int_0^{h-v} \min(s,\epsilon) \de s \right)\dx \dt \\
		& \geq \frac12 \int_{B_2(0)} \phi^2(x,\tau) \dx \,.
	\end{align*}
	We also have the elementary inequality
	\begin{equation*}
		\big ( |\grad h|^{p-2} \grad h - |\grad v|^{p-2}\grad v \big ) \cdot \grad \phi \geq \frac1C |\nabla \phi|^p\,,
	\end{equation*}
	since $p>2$. Hence by combining the last three displays we arrive at
	\begin{align*}
		\sup_{0 < t < T_0} \int_{B_{1}(0)} \phi^2(x, t) \dx + \int_{Q} |\grad \phi|^p \dx \dt \leq C \epsilon \mu_v (Q)\,.
	\end{align*}
	Using the parabolic Sobolev inequality (\cite[Corollary I.3.1]{DB}) we obtain that
	\begin{align*}
		\epsilon^p |\tilde Q| & \leq \int_Q \phi^p \dx \dt \\
		& \leq C \left (\sup_{0 < t < T_0} \int_{B_2(0)} \phi^p(x, t) \dx + \int_{Q} |\grad \phi|^p \dx \dt \right ) \\
		&\leq C \left (\epsilon^{p-2}\sup_{0 < t < T_0} \int_{B_2(0)} \phi^2(x, t) \dx + \int_{Q} |\grad \phi|^p \dx \dt \right )  \,.
	\end{align*}
	Hence we see that
	\begin{equation*}
		1 \leq C\, \mu_v (Q)
	\end{equation*}
	with a constant $C \equiv C(p,n,M)$ through the dependencies of $\epsilon, r_1,T_0$.
\end{proof}

The next lemma provides a comparison estimate for the measures. If two solutions are ordered, then the corresponding measures will be ordered as well.

\begin{lemma}
	\label{meas order} Let $\Omega \subset \R^n$ be a domain. Let $u$ and $v$ be weak solutions in $(\Omega \cap B_r(0)) \times (0,T)$ such that $u \geq v \geq 0$ and both vanish continuously on the lateral boundary $(\partial \Omega \cap B_r(0)) \times (0,T)$. Then
	\begin{equation*}
		\mu_v \leq \mu_u \qquad \mbox{in } \, B_r(0) \times (0,T)
	\end{equation*}
	in the sense of measures.
\end{lemma}
\begin{proof}
	To show this consider the test function $\phi = \min(1, (u-v-\epsilon)_+/\epsilon) \psi$, where $\psi$ is non-negative and belongs to $C_0^\infty (Q)$ with $Q = B_r(0) \times (0,T)$. Obviously $\phi$ is supported in $(\Omega \cap B_r(0)) \times (0,T)$, because both $u$ and $v$ vanish continuously on the lateral boundary $(\partial \Omega \cap B_r(0)) \times (0,T)$. Since both $u$ and $v$ are weak solutions, we have, by extending them both by zero in $(B_r(0) \setminus \Omega) \times (0,T)$, that
	\begin{align}\label{eq:diffmeasure}
		\int_Q (u-v)_t \phi \dx \dt + \int_Q \big (|\grad u|^{p-2}\grad u -|\grad v|^{p-2}\grad v \big ) \cdot \grad \phi \dx \dt = 0\,.
	\end{align}
	Let us first treat the time term. Integrating by parts we get that
	\begin{align}
		\int_Q (u-v)_t \phi \dx \dt & = \int_{Q} \partial_t \left( \int_0^{u-v} \min(1, (s-\epsilon)_+/\epsilon) \de s \right) \psi \dx \dt \notag \\
		& = - \int_{Q} \left( \int_0^{u-v} \min(1, (s-\epsilon)_+/\epsilon) \de s \right) \partial_t \psi\dx \dt \notag \\
		& \to - \int_{Q} (u-v) \partial_t \psi\dx \dt\,,  \label{eq:dtconvergence}
	\end{align}
	as $\epsilon \to 0$.
	To treat the elliptic term, we begin by noting that
	\begin{align} \label{eq:phigradient}
		\grad \phi = \grad \psi \min(1, (u-v-\epsilon)_+/\epsilon) + \frac{1}{\epsilon} \psi \grad (u-v) \chi_{U_\epsilon}\,,
	\end{align}
	where $U_\epsilon:= \{u-v > \epsilon \} \cap (\Omega \cap B_r(0)) \times (0,T)$. The second term in \Cref{eq:phigradient} will give rise to a positive term in \Cref{eq:diffmeasure}, hence we discard it and obtain an inequality
	\begin{align}
		& \int_Q \big (|\grad u|^{p-2}\grad u -|\grad v|^{p-2}\grad v \big ) \cdot \grad \phi \dx \dt \notag \\
		& \qquad \geq \int_Q \big (|\grad u|^{p-2}\grad u -|\grad v|^{p-2}\grad v \big ) \cdot \grad \psi \min(1, (u-v-\epsilon)_+/\epsilon) \dx \dt \notag \\
		& \qquad \to \int_Q \big (|\grad u|^{p-2}\grad u -|\grad v|^{p-2}\grad v \big ) \cdot \grad \psi \dx \dt\,,  \label{eq:dxconvergence}
	\end{align}
	as $\epsilon \to 0$ by dominated convergence. Combining the convergence in \Cref{eq:dtconvergence,eq:dxconvergence} with \Cref{eq:diffmeasure} we arrive at the inequality
	\begin{equation}
		\label{eq:meascomp} - \int_{Q} (u-v) \partial_t \psi\dx \dt + \int_Q \big (|\grad u|^{p-2}\grad u -|\grad v|^{p-2}\grad v \big ) \cdot \grad \psi \dx \dt \leq 0
	\end{equation}
	after sending $\epsilon \to 0$. Since the non-negative function $\psi \in C_0^\infty (Q)$ was arbitrary, \Cref{eq:meascomp} finishes the proof after recalling the definitions of $\mu_u$ and $\mu_v$.
\end{proof}

We now have all the technical tools to complete the proof of \Cref{MuLowerBound}.
\begin{proof}
	[Proof of \Cref{MuLowerBound}] Let $u$ be as in \Cref{MuLowerBound} with $A_{r}^- := (a_{r/2}(x_0),t_0)$. Applying the Harnack estimate in~\Cref{NTA FChain} yields for a constant $C = C(p,n,M)$
	\begin{equation*}
		u(A_r^-) \leq \tilde C \inf_{y \in B_{r/(8M)}(a_{r/2}(x_0))} u\left(y,t_0 + \tau_0 u(A_{r}^-)^{2-p} r^p \right)\,, \qquad \tau_0 := \frac{C}{(2M)^p}\,,
	\end{equation*}
	since
	\begin{equation*}
		t_0 - \tau_0 u(A_{r}^-)^{2-p} r^p >0 \,.
	\end{equation*}
	Consider the scaled solution
	\begin{equation*}
		\hat u (x,t ) = \frac{1}{\lambda} u \left (x_0 + \frac{r}{2} x, t _0 + \tau_0 u(A_{r}^-)^{2-p} r^p + \lambda^{2-p} \left(\frac{r}{2} \right)^p t \right)\,, \qquad \lambda:= \frac{u(A_r^-) }{\tilde C}\,,
	\end{equation*}
	so that
	\begin{equation*}
		\inf_{y \in B_{1/(4M)}(a_1(0))} \hat u(y,0) \geq 1\,,
	\end{equation*}
	and $\Omega$ is mapped to $\hat \Omega$ and $(x_0,t_0 + \tau_0 u(A_{r}^-)^{2-p} r^p )$ to $(0,0)$. The comparison principle shows that the function $v$ defined in~\Cref{MuLowerSimple} satisfies $v\leq \hat u$ in $(\hat \Omega \cap B_2(0)) \times (0,T_0)$ provided that
	\begin{equation*}
		t _0 + \tau_0 u(A_{r}^-)^{2-p} r^p + \lambda^{2-p} \left(\frac{r}{2} \right)^p T_0 < T\,.
	\end{equation*}
	Thus we choose $\tau_1 := \tilde C^{p-2} 2^{-p} T_0$ in the statement. Applying~\Cref{MuLowerSimple,meas order}
	\begin{equation*}
		\mu_{\hat u}\big(B_2(0) \times (0,T_0)\big) \geq \mu_v\big(B_2(0) \times (0,T_0)\big) \geq 1/C \,.
	\end{equation*}
	Scaling back to $u$ gives us the result.
\end{proof}


\section{Construction of barriers} \label{s.barriers}

In this section we construct the barriers that will serve as the starting point for the estimates of the decay-rate of the solutions. The upper barrier in \Cref{UpperBarrier} is based on the function constructed in \cite[Theorem 4.1]{DKV}. However the subsolution constructed in \Cref{lembar} seems to be new and allows us to obtain $p$-stable estimates from below on the decay rate.
\begin{lemma}
	\label{lembar} Let $T = (n^{p-1} p)^{-1}$ and for $a\in (0,1)$, let
	\begin{align*}
		\label{rho_0} \varrho_0&:= \min\left \{ \frac{a p}{n (p-2)} + 1 \, , 2\right\}^{\frac{p-1}{p}} \,.
	\end{align*}
	Then the function $h$,
	\begin{align*}
		h(x,t)= g(|x|,t) - g(\varrho_0,t), \quad g(r,t) = \left [1-\frac{p-2}{p^{\frac{p}{p-1}}} \frac{r^{\frac{p}{p-1}}-1}{t^{\frac{1}{p-1}}}\right ]_+^{\frac{p-1}{p-2}},
	\end{align*}
	is a classical subsolution in $(B_{\varrho_0}(0)\setminus \overline {B_{1}(0)})\times (0,T)$ satisfying the boundary conditions
	\begin{equation} \label{eq:hbdrycond}
		\begin{cases}
			h = 0, &\text{on }\partial B_{\varrho_0}(0)\times (0,T) \\
			h = 0, &\text{on } (B_{\varrho_0}(0)\setminus \overline {B_{1}(0)})\times \{0\} \\
			h = 1-\left [1-\frac{p-2}{p^{\frac{p}{p-1}}} \frac{{\varrho_0} ^{\frac{p}{p-1}}-1}{t^{\frac{1}{p-1}}}\right ]_+^{\frac{p-1}{p-2}} , & \text{on } \partial B_{1}(0) \times (0,T)\,.
		\end{cases}
	\end{equation}
	Furthermore, $x \mapsto h(x,t)$ is a radially decreasing function satisfying
	\begin{equation*}
		\inf_{1 \leq |x| \leq {\varrho_0} } |\grad h(x, T)| \geq n \exp\left(-\frac{n}{p} \right) (1-a)^{\frac{n}{p}} \,,
	\end{equation*}
	and $h(x,t) \equiv h_p(x,t)$ tends to
	\begin{equation*}
		\exp\left (-\frac{|x|^2-1}{4t}\right )-\exp \left(-\frac{1}{4t}\right )
	\end{equation*}
	as $p \to 2^+$.
\end{lemma}
\begin{remark}
	Note that the function in \Cref{lembar} is not continuous up to the boundary at the corner $\partial {B_{1}(0)} \times \{0\}$. However, the \emph{limsup} as we approach a point on this piece from the inside of $D := (B_{\varrho_0}(0)\setminus \overline {B_{1}(0)})\times (0,T)$ is $1$ for $h$. This implies, see \cite[Lemma 4.4]{KL}, that if we have a weak supersolution $u$ in $D$, staying above the boundary conditions in \Cref{eq:hbdrycond} in \emph{liminf} sense, and staying above $1$ on the corner $\partial {B_{1}(0)} \times \{0\}$ again in the \emph{liminf} sense, then $u$ will be above $h$ in $D$.
\end{remark}
\begin{proof}
	Let $h$ and $\varrho_0$ be as stated. By construction, the boundary conditions for $h$ are in force. To verify that $h$ is a classical subsolution in $(B_{\varrho_0}(0)\setminus \overline {B_{1}(0)})\times (0,T)$, we first compute,
	\begin{align*}
		\nabla g(|x|,t) & = - \frac{x}{|x|} \left[ \frac1p \frac{|x|}{t} g(|x|,t) \right]^{\frac{1}{p-1}} \,, \\
		|\nabla g(|x|,t)|^{p-2} \nabla g(|x|,t) & = -\frac1p \frac{x}{t} g(|x|,t)\,, \\
		-\Delta_p g(|x|,t) & = \left[\frac{n}{p t} g(|x|,t)^{\frac{p-2}{p-1}} - p^{-\frac{p}{p-1}} \frac{|x|^{ \frac{p}{p-1} }}{t^{\frac{p}{p-1}}} \right] g(|x|,t)^{\frac{1}{p-1}}\,, \\
		\partial_t g(r,t) & = p^{-\frac{p}{p-1}} \frac{r^{ \frac{p}{p-1} }-1}{t^{\frac{p}{p-1}}} g(r,t)^{\frac{1}{p-1}}\,.
	\end{align*}
	Observing that $ \partial_t h(x,t) \leq \partial_t g(|x|,t)$ it is enough to verify $g_t - \Delta_p g \geq 0$ in $(B_{\varrho_0}(0)\setminus \overline {B_{1}(0)})\times (0,T)$ for $g > 0$. Assuming $g > 0$,  we see, since $g(|x|,t) \leq 1$ for $|x| > 1$, that
	\begin{align*}
		\frac{(h_t-\Delta_p h)(x,t)}{g(|x|,t)^{\frac{1}{p-1}}} & \leq \left[ p^{-\frac{p}{p-1}} \frac{|x|^{ \frac{p}{p-1} }-1}{t^{\frac{p}{p-1}}} + \frac{n}{ p t} g(|x|,t)^{\frac{p-2}{p-1}} - p^{-\frac{p}{p-1}} \frac{|x|^{ \frac{p}{p-1} }}{t^{\frac{p}{p-1}}} \right] \\
		& \leq \frac{1}{ p t} \left[ n - \left(\frac{1}{ p t} \right)^{\frac{1}{p-1}} \right]\,.
	\end{align*}
	Since $ 0< t < T = (n^{p-1} p)^{-1}$ we have that $h_t-\Delta_p h\leq 0$ for $|x| > 1$ and $t \in (0,T)$. Note also that our choices of parameters are stable as $p \to 2$. Next, by yet an other explicit calculation we obtain
	\begin{equation*}
		\label{gradbd} \inf_{1 \leq |x| \leq \varrho_0} |\grad h(x,T)| \geq \left( \frac{1}{p T} g(\varrho_0,T) \right)^{\frac{1}{p-1}}= n g(\varrho_0,T)^{\frac{1}{p-1}}\,.
	\end{equation*}
	To complete the proof we need to estimate $g(\varrho_0,T)$ from below. To do this we note that
	\begin{equation*}
		 g(\varrho_0,T)^{\frac{1}{p-1}} = \left[ 1 - \frac{n(p-2)}{p} (\varrho_0^{ \frac{p}{p-1} }-1 ) \right]^{\frac{1}{p-2}}\,,
	\end{equation*}
	and we consider two cases. First, if $\varrho_0=2^\frac{p-1}{p}$, then $ap \geq n(p-2)$ and
	\begin{equation*}
		 g(\varrho_0,T)^{\frac{1}{p-1}} = (1-s)^{b/s}\,, \quad s = \frac{n(p-2)}{p}\,, \quad b = \frac{n}{p}\,.
	\end{equation*}
	Furthermore, for $\eta \in [0,1)$ we have
	\begin{align}
		(1-\eta)^{\frac{b}{\eta}} & = \exp\left(-\frac{b}{\eta} \sum_{k=1}^\infty \frac{\eta^k}{k} \right) = \exp \left (-b \sum_{k=1}^\infty \frac{\eta^{k-1}}{k}\right ) = \exp\left (-b - b \sum_{k=1}^\infty \frac{\eta^{k}}{k+1}\right) \notag \\
		& \geq \exp\left (-b-b\sum_{k=1}^\infty \frac{\eta^k}{k}\right)  = e^{-b}(1-\eta)^b \,. \label{eq:barrier helper}
	\end{align}
	Since $ap \geq n(p-2)$ implies $s \leq a < 1$, we can apply \Cref{eq:barrier helper} to get
	\begin{equation*}
		 g(\varrho_0,T)^{\frac{1}{p-1}} \geq \exp(-b) (1-s)^{b} \geq \exp(-b) (1-a)^{b} \,.
	\end{equation*}
	Second, if $\varrho_0<2^\frac{p-1}{p}$, then $ap < n(p-2)$ and using \Cref{eq:barrier helper} we get
	\begin{equation*}
		 g(\varrho_0,T) = \left( 1 - a \right)^{\frac{p-1}{p-2}} \geq \left( 1 - a \right)^{b/a} \geq \exp(-b) (1-a)^{b}.
	\end{equation*}
	Collecting the results of the two cases completes the proof of the lemma.
\end{proof}
\begin{remark}
	Note that we could, as in \cite{SV}, instead of the function in \Cref{lembar} use the Barenblatt fundamental solution together with the barriers from \cite{BV} to establish a version of \Cref{lembar}. However, this would result in a construction which is not $p$-stable.
\end{remark}

In the next lemma we construct a certain supersolution to be used in the subsequent arguments.

\begin{lemma}
	\label{UpperBarrier} Let $T,H > 0$ be given degrees of freedom. Let
	\begin{equation*}
		 k \in (0,k_0]\,, \qquad k_0 := \min\left\{\frac{p-1}{n}, T^{\frac{1}{p-1}} H^{\frac{p-2}{p-1}} \right\}\,.
	\end{equation*}
	There exists a classical supersolution $\tilde h$ in
	\begin{equation*}
		N=\{(x,t): 1<|x|<1+k,\ 0<t<T\}\,,
	\end{equation*}
	such that
	\begin{equation*}
		\begin{cases}
			\tilde h \geq 0, &\text{ in } \partial B_{1}(0) \times (0,T]\\
			\tilde h \geq H &\text{ on } (\overline {B_{1+k}(0)} \setminus B_{1}(0)) \times \{0\}\\
			\tilde h \geq H &\text{ on } \partial B_{1+k}(0) \times [0,T]\,,
		\end{cases}
	\end{equation*}
	and such that
	\begin{equation} \label{tilde h upper}
		\tilde h(x,T) \leq \frac{H\exp(2)}{k} (|x|-1),
	\end{equation}
	whenever $x \in B_{1+k}(0) \setminus B_{1}(0)$.
\end{lemma}
\begin{proof}
	This type of construction was originally carried out in \cite[Theorem~4.1]{DKV} and we here include a proof for completion. Let
	\begin{equation*}
		v(x,t)= \exp\left( \frac{t-T}{T} - \frac{|x|-1}{k} \right)\,,
	\end{equation*}
	and let
	\begin{equation*}
		\tilde h(x,t)= \tilde H (1-v(x,t))\,, \qquad \tilde H = H \exp(2)\,,
	\end{equation*}
	accordingly. Then
	\begin{equation*}
		\partial_t \tilde h(x,t) = -\frac{\tilde H}{T}v(x,t) \qquad \mbox{and} \qquad \nabla \tilde h(x,t) = \frac{\tilde H x}{k|x|} v(x,t)\,.
	\end{equation*}
	Observe also that $\tilde H v \geq H$ for all $(x,t) \in N$ and that $\tilde h$ satisfies the boundary conditions. We now show that $\tilde h$ is a classical supersolution in $N$. Indeed, by a straightforward calculation we see that
	\begin{align*}
		\tilde h_t - \Delta_p \tilde h &= - \frac{\tilde H}{T} v + \tilde H^{p-1} v^{p-1} \left ( (p-1)k^{-p} - k^{1-p} \frac{n-1}{|x|} \right ) \\
		& \geq \frac{\tilde H}{T} v \left (-1 + T H^{p-2} k_0^{1-p} + T H^{p-2} k^{1-p} ( (p-1) k_0^{-1} - n ) \right) \\
		& \geq 0\,,
	\end{align*}
	whenever $(x,t) \in N$. Finally, since
	\begin{equation*}
		 \sup_{1<|x|<1+k}|\nabla \tilde h(x,T)| \leq \frac{\tilde H}{k} \qquad \text{and} \qquad \tilde h(x,T) = 0 \quad \forall x \in \partial B_{1}(0)\,,
	\end{equation*}
	we obtain the upper bound for $\tilde h(x,T)$ as well.
\end{proof}


\section{Decay estimates and a change of variables}  \label{s.decay}


In this section we prove a lower bound (\Cref{ComparisonFunction}) and an upper bound (\Cref{SUniversalUpperBound1}) on the decay of solutions. The following lemma, which is a change of variables, will be used in the proof of our decay estimates. The proof of the lemma follows from \cite[Lemma 3.5]{Ku}.
\begin{lemma}
	\label{lemchange} Let $u=u(x,t)$ be $p$-parabolic in $\Omega\times (T_0,T_1)$. Let $C > 0$ be given and let
	\begin{align*}
		\label{aa1} f(\tau) &= \frac{1}{C(p-2)} \left ( \exp(C(p-2)\tau) - 1 \right)\,,\\
		g(\eta) & = (C(p-2)\eta + 1)^{\frac{1}{p-2}}\,,
	\end{align*}
	for $p>2$, and
	\begin{align*}
		f(\tau)= \tau,\ g(\eta) = \exp(C\eta)\,,
	\end{align*}
	for $p=2$. Let $w(x,\tau) = g(f(\tau)) u(x,f(\tau))$. Then $w(x,\tau)$ is a (weak) solution to the equation,
	\begin{equation*}
		\partial_\tau w = \Delta_p w + C w
	\end{equation*}
	in $\Omega\times(\tau_0,\tau_1)$ where $f(\tau_i)=T_i$, $i\in\{0,1\}$.
\end{lemma}

\subsection{A lower bound on the decay} Using the classical subsolution constructed in \Cref{lembar} and the change of variable outlined in \Cref{lemchange} we here prove the following lemma, which describes the optimal decay rate from below after a certain intrinsic waiting time. This lemma will be crucial when proving global $C^{1,1}$-estimates, see \Cref{s.global} and when proving the local $C^{1,1}$-estimates, see \Cref{s.local}.

\begin{lemma}
	\label{ComparisonFunction} Let $0<\varrho \leq r/4$ and let $g \in L^2(B_{r}(x_0))$ be a non-negative function satisfying
	\begin{equation*}
		 \fint_{B_\varrho(x_0)} g \dx \geq \lambda>0\,.
	\end{equation*}
	Assume that $\hat h \in C([t_0,\infty);L^2(B_{r}(x_0)))$ is a weak non-negative solution solving the Cauchy problem
	\begin{equation*}
		\begin{cases}
			\hat h_t - \Delta_p \hat h = 0 &\text{ in }B_{r}(x_0)\times (t_0,\infty)\\
			\hat h = g &\text{ on } B_{r}(x_0) \times \{t_0\}\,.
		\end{cases}
	\end{equation*}
	Then there exist constants $c_i \equiv c_i (\varrho/r,n,p)$, $i\in \{1,2\}$, such that
	\begin{equation*}
		\hat h(x,t) \geq \frac{\lambda}{c_1} \left ( c_1(p-2) \frac{t-t_0}{\lambda^{2-p}r^p} +1 \right )^{- \frac{1}{p-2}} \frac{d(x, \partial B_{r}(x_0))}{r}\,,
	\end{equation*}
	whenever $(x,t) \in B_{r}(x_0) \times (t_0 + c_2 \lambda^{2-p}r^p,\infty)$. Furthermore, constants $c_i$, $i \in \{1,2\}$, are stable as $p \to 2^+$.
\end{lemma}
\begin{proof}
	After scaling we may assume that $x_0=0$, $t_0=0$, $r=1$, and $\lambda=1$. Let
	\begin{align*}
		\label{rho_0 2} \varrho_0&:= \min\left \{ \frac12 \frac{p}{n (p-2)} + 1 \,, 2\right\}^{\frac{p-1}{p}}\,.
	\end{align*}
	Applying \Cref{FChainMeas} we find a time $t^\filledstar \equiv t^\filledstar(n,p,\varrho/r)$ and a constant $c^\filledstar \equiv c^\filledstar(n,p,\frac{\varrho}{r})$ such that
	\begin{equation*}
		 \hat h(x,t^\filledstar) \geq 1/c^\filledstar \qquad \forall x \in B_{1/\varrho_0}(0)\,.
	\end{equation*}
	Set $\bar h(x,t) = c^\filledstar \hat h(x/\varrho_0, t^\filledstar + t [c^{\filledstar}]^{p-2}/\varrho_0^p )$ and let $w(x,\tau) = g(f(\tau)) \bar h(x, f(\tau))$, where $g$ and $f$ are defined as in \Cref{lemchange}. Then $w(x,\tau)$ is a non-negative weak solution to the equation
	\begin{equation*}
		\partial_\tau w = \Delta_p w + C w
	\end{equation*}
	in $B_{\varrho_0}(0)\times (0,\infty)$ and $w(x,0) \geq 1$ for all $x \in B_{1}(0)$. In particular, $w$ is a weak supersolution in $B_{\varrho_0}(0)\times (0,\infty)$. Now, \Cref{TuomoWH}, \cite[Corollary 3.6]{Ku}, \Cref{lemchange}, and \cite[Proposition 3.1]{Ku} imply that we have, for a new constant $\bar c \equiv \bar c(n,p) > 1$,
	\begin{equation}
		\label{lowbdd} w (x,\tau) \geq 1/\bar c\,, \quad (x,\tau) \in B_{1}(0) \times (0,\infty)\,,
	\end{equation}
	provided we choose $C$ large enough in the definitions of $f$ and $g$ in \Cref{lemchange}. Consider $\hat \tau \geq 0$ arbitrary, let $h$ be the classical subsolution of \Cref{lembar} and let $T$ be as in \Cref{lembar}. Then, simply using the intrinsic scaling, the comparison principle, and \Cref{lowbdd} we see that
	\begin{equation} \label{eq:lowforeverbound}
		w(x,\tau) \geq \frac{1}{\bar c}h(x,\bar c^{2-p}(\tau - \hat \tau))
	\end{equation}
	whenever $(x,\tau)\in (B_{\varrho_0}(0) \setminus B_{1}(0)) \times (\hat \tau,\hat \tau+\bar c^{p-2} T)$. Since $\hat \tau \geq 0$ was arbitrary we get from \Cref{eq:lowforeverbound,lembar} that there is a $c \equiv c(n,p)$ such that
	\begin{equation}
		\label{bound} w(x,\tau) \geq \frac{1}{c} d(x, \partial B_{\varrho_0}(0))\,, \quad (x,\tau) \in B_{\varrho_0}(0) \times (\bar c^{p-2} T,\infty)\,.
	\end{equation}
	To complete the proof, it suffices to rephrase \Cref{bound} in terms of $\hat h(x,t)$.
\end{proof}


\subsection{An upper bound on the decay}

Working with solutions vanishing on the entire lateral boundary, we will make use of the following decay estimates.
\begin{lemma}
	\label{Tuomo1} Let $u \in C ([0,T);L^2(\Omega)) \cap L^{p}(0,T; W_{0}^{1,p}(\Omega))$ be a non-negative weak subsolution in $\Omega \times (0,T)$. Then there exists a constant $c=c(n,p,|\Omega|)$ such that
	\begin{equation*}
		\sup_{\Omega} u(\cdot, t) \leq \frac{c}{t^{n/\sigma}} \left (\fint_{\Omega} u(x,0) \dx \right )^{p/\sigma},\quad \sigma= n(p-2) + p\,.
	\end{equation*}
	The constant $c$ is stable as $p \to 2^+$.
\end{lemma}
\begin{proof}
	See \cite[Lemma 3.24]{KuPa} and use the $L^1$-contractivity (test with the function $\min(1,u/\epsilon)$) and the comparison principle.
\end{proof}
In the following lemma we describe the optimal decay of the supremum of a solution vanishing on the entire lateral boundary, it follows from an iterative rescaling and comparison together with the decay estimate in \Cref{Tuomo1}.
\begin{lemma}
	\label{SUniversalUpperBound1} Let $u \in C ([0,T);L^2(\Omega))  \cap L^{p}(0,\infty; W_{0}^{1,p}(\Omega))$ be a non-negative weak subsolution in $\Omega \times (0,\infty)$. Then there exist constants $c_i\equiv c_i(p,n,|\Omega|)$, $i\in\{1,2\}$, such that the following holds. Let
	\begin{equation*}
		\bar\Lambda := \fint_{\Omega} u(x,0) \dx \,,
	\end{equation*}
	then
	\begin{equation*}
		\sup_{\Omega} u(\cdot, t) \leq c_1 \left ((p-2)\frac{\bar\Lambda^{p-2} }{c_1} t +1 \right )^{-\frac{1}{p-2}} \bar\Lambda \,,
	\end{equation*}
	whenever $t> c_2 \bar\Lambda^{2-p}$. The constants $c_i$, $i\in\{1,2\}$, are stable as $p \to 2^+$.
\end{lemma}
\begin{proof}
	Let $w=w(x,t)$ solve the Dirichlet problem
	\begin{equation}
		\label{SimpleDirichlet}
		\begin{cases}
			w_t - \Delta_p w = 0 &\text{ in }\Omega \times (0,\infty)\\
			w = 0 &\text{ on } \partial \Omega \times [0,\infty)
		\end{cases}
	\end{equation}
	and assume that
	\begin{align}
		\label{norm}\fint_{\Omega} w(x,0) \dx\leq 1 \,.
	\end{align}
	Applying \Cref{Tuomo1} to $w$ we see that
	\begin{equation*}
		\label{Tstarsmaller1} \sup_{\Omega} w(\cdot, t) \leq \frac{c}{t^{n/\sigma}} \,,
	\end{equation*}
	for some $c=c(n,p,|\Omega|)$, and for all $t>0$. In particular, there exists $t^\filledstar=t^\filledstar(n,p,|\Omega|)>0$ such that
	\begin{equation}
		\label{Tstarsmaller} \sup_{\Omega} w(\cdot, t^\filledstar) \leq \frac12\,.
	\end{equation}
	To prove the lemma we will now use \Cref{Tstarsmaller} in an iterative argument. In particular, consider the function
	\begin{equation*}
		w_1(x,t) := \bar\Lambda^{-1} u (x,\bar\Lambda^{2-p} t) \,.
	\end{equation*}
	Then $w_1$ is a solution to \Cref{SimpleDirichlet} satisfying \Cref{norm}. Hence we have by \Cref{Tstarsmaller} that
	\begin{equation*}
		w_1(x,t^\filledstar) \leq \frac{1}{2} \,,
	\end{equation*}
	which after scaling back becomes
	\begin{equation*}
		u(x, \bar\Lambda^{2-p} t^\filledstar) \leq 2^{-1}\bar\Lambda \quad \text{whenever} \quad x\in\Omega \,.
	\end{equation*}
	Next, consider the function
	\begin{equation*}
		w_2(x,t) := (2^{-1} \bar\Lambda)^{-1}{u} (x, \bar \Lambda^{2-p} t^\filledstar + (2^{-1} \bar\Lambda)^{2-p} t) \,,
	\end{equation*}
	which again satisfies \Cref{SimpleDirichlet,norm}.
	Applying \Cref{Tstarsmaller} to the function $w_2$ we deduce, by elementary manipulations, that
	\begin{equation*}
		\sup_{\Omega} u(\cdot, (\bar\Lambda^{2-p}+(2^{-1} \bar\Lambda)^{2-p}) t^\filledstar) \leq 2^{-2}\bar\Lambda \,.
	\end{equation*}
	Proceeding inductively we deduce that
	\begin{equation*}
		\sup_{\Omega} u(\cdot, t_j) \leq 2^{-j} \bar\Lambda \,,
	\end{equation*}
	where
	\begin{equation*}
		t_j := t^\filledstar \bar\Lambda^{2-p} \sum_{k=0}^{j-1} 2^{k(p-2)}\,,\quad j\in\{1,2,...\}\,.
	\end{equation*}
	To complete the argument consider $t\in (t_1,\infty)$ and let $j$ be the largest $j$ such that $t_j\leq t$. Then, by the comparison principle and by construction,
	\begin{equation*}
		\sup_{x\in \Omega} u(\cdot, t) \leq 2^{-j} \bar\Lambda
	\end{equation*}
	and
	\begin{equation*}
		t_j\leq t<t_{j+1}\,.
	\end{equation*}
	Since
	\begin{equation*}
		 2^{-j} = \left(\frac{(2^{p-2}-1)}{t^\filledstar} \bar \Lambda^{p-2} t_j + 1 \right)^{-\frac1{p-2}}\,,
	\end{equation*}
	and $2^{p-2}-1 \geq \log(2) (p-2) $, by retracing the argument we derive the conclusion of the lemma. Furthermore, the constants $c_i$, $i\in\{1,2\}$, are stable as $p \to 2^+$. In particular, we see that
	\begin{equation*}
		\left ((p-2)c_1^{-1}(n,p,|\Omega|) \bar\Lambda^{p-2} t+1 \right )^{-\frac{1}{p-2}} \to \exp \left (-c_1^{-1}(n,2,|\Omega|) t \right )\,.
	\end{equation*}
	This completes the proof of the lemma.
\end{proof}


\section{Global estimates in $C^{1,1}$-domains} \label{s.global}
In this section we combine the optimal decay estimate established in \Cref{SUniversalUpperBound} together with the barrier function in \Cref{UpperBarrier}, to obtain the sharp decay estimate from above. Note that taking the initial data to be $+\infty$ allows us to see that this is sharp with respect to the so called ``friendly giant'', see for example \cite{KLP}. 
\begin{lemma}
	\label{SUniversalUpperBound} Let $\Omega\subset\R^n$ be a bounded $C^{1,1}$-domain satisfying the ball condition with radius $r_0$. Let $u \in C ([0,\infty);L^2(\Omega)) \cap L^{p}(0,\infty; W_{0}^{1,p}(\Omega))$ be a non-negative $p$-parabolic function in $\Omega \times (0,\infty)$. Let
	\begin{equation*}
		\bar\Lambda := \fint_{\Omega} u(x,0) \dx\,.
	\end{equation*}
	Then there exists $C_i \equiv C_i(\diam \Omega/r_0,n,p)$, $i \in \{1,2\}$, such that
	\begin{equation*}
		u(x, t) \leq C_1 \bar\Lambda \left ((p-2)C_1^{-1} \bar\Lambda^{p-2} t +1 \right )^{ -\frac{1}{p-2}} d(x, \partial \Omega)
	\end{equation*}
	whenever $t > C_2 \bar\Lambda^{2-p}$. Furthermore, constants $C_i$, $i \in \{1,2\}$, are stable as $p \to 2^+$.
\end{lemma}
\begin{proof}
	By scaling we can without loss of generality assume that $\bar\Lambda=1$. Let $\hat x_0 \in \partial \Omega$ be an arbitrary point. Assume, for simplicity that $a^-_{2r_0}(\hat x_0) = 0$, where $a^-_{2r_0}(\hat x_0)$ is the exterior corkscrew point as in \Cref{ballNTAcork}.
	Consider an arbitrary number $\hat t$, such that $\hat t > c_2$ (where $c_2$ is from \Cref{SUniversalUpperBound1}). Then, using \Cref{SUniversalUpperBound1} we see that
	\begin{align}\label{eq:t hat upper}
		u(x,t) \leq c_1 \left (\frac{p-2}{c_1} \hat t+1 \right )^{\frac{1}{2-p}} =: \bar\Lambda_u(\hat t)\,, \qquad x \in \Omega\,, t \geq \hat t\,.
	\end{align}
	Construct the function $\tilde h$ in \Cref{UpperBarrier} with the choices ($T = \bar \Lambda_u^{2-p}$, $H = \bar \Lambda_u$), then $k_0$ from \Cref{UpperBarrier} simplifies to
	\begin{equation*}
		k_0 = \min \left \{ \bar \Lambda_u^{\frac{2-p}{p-1}} \bar \Lambda_u^{\frac{p-2}{p-1}},\frac{p-1}{n} \right \} = \min \left \{1,\frac{p-1}{n} \right \}\,.
	\end{equation*}
	Consider now the function $\hat h$ defined as follows
	\begin{equation} \label{eq:h hat}
		\hat h(x,t) = \tilde h\left ( \frac{x}{r_0}, \frac{t-\hat t}{r_0^p}\right )\,.
	\end{equation}
	The function $\hat h$ is a supersolution in
	\begin{equation*}
		N:=(B_{(1+1/k_0)r_0} \setminus \overline B_{r_0}) \times (\hat t, \hat t + \bar \Lambda_u^{2-p} r_0^p)\,.
	\end{equation*}
	Thus, the comparison principle, \Cref{defBall}, \Cref{eq:t hat upper,eq:h hat} imply that
	\begin{equation*}
		u(x,t) \leq \hat h(x,t) \quad \text{in} \quad N \cap \Omega_\infty\,.
	\end{equation*}
	Next, using the upper estimate \Cref{tilde h upper} from \Cref{UpperBarrier} we see that
	\begin{equation} \label{boundary upper}
		u(x,\hat t + \bar \Lambda_u^{2-p} r_0^p) \leq C \bar \Lambda_u  \left ( \left |\frac{x}{r_0} \right |-1 \right)\,,
	\end{equation}
 for a constant $C = C(n,p)$. As $\hat x_0 \in \partial \Omega$ is arbitrary, we see, using from \Cref{boundary upper,eq:t hat upper}, that
	\begin{equation*}
		u(x,\hat t + \Lambda_u^{2-p} r_0^p) \leq C d(x,\partial \Omega)\,,
	\end{equation*}
for a new constant $C \equiv C(\diam \Omega/r_0,n,p)$. Furthermore, as $\hat t > c_2$ was arbitrary, we see that if $t > \bar \Lambda_u(c_2)^{2-p} r_0^p+c_2:=C_2$, then
	\begin{equation*}
		u(x,t) \leq C_1 \left ((p-2)C_1^{-1} t +1 \right )^{ -\frac{1}{p-2}} d(x, \partial \Omega)\,,
	\end{equation*}
for a constant $C_1 \equiv C_1(\diam \Omega/r_0,n,p) > 1$.  This completes the proof of the lemma.
\end{proof}
The next lemma provides the corresponding lower bound.
\begin{lemma}
	\label{SUniversalLowerBound} Let $\Omega\subset\R^n$ be a bounded connected $C^{1,1}$-domain satisfying the ball condition with radius $r_0$. Let $u \in C ([0,\infty);L^2(\Omega))$ be a non-negative $p$-parabolic function in $\Omega \times (0,\infty)$. Suppose that there is a ball $B_{4r_1}(x_1) \subset \Omega$, $r_1 \in (0,r_0)$, such that
	\begin{equation*}
		 \bar\Lambda := \fint_{B_{r_1}(x_1)} u(x,0) \dx > 0\,.
	\end{equation*}
	Then there exist $C_i \equiv C_i(\frac{\diam \Omega}{r_0},\frac{r_1}{r_0},n,p)$, $i \in \{3,4\}$, such that
	\begin{equation*}
		u(x, t) \geq \frac{\bar \Lambda}{C_3} \left ((p-2)C_3 \bar\Lambda^{p-2} t +1 \right )^{-\frac{1}{p-2}} d(x, \partial \Omega)\,,
	\end{equation*}
	whenever $x\in \Omega$ and $t > C_4 \bar \Lambda^{2-p}$. Furthermore, constants $C_i$, $i \in \{3,4\}$, are stable as $p \to 2^+$.
\end{lemma}
\begin{proof}
	After scaling and translating we may assume that $\bar\Lambda=1$, $x_1 = 0$ and $r_0=1$. Note that with these assumptions we have
	\begin{equation}
		\label{tr} \fint_{B_{r_1}(0)} u(x,0) \dx =1 \,.
	\end{equation}
	Denote now the set
	\begin{equation*}
		 \Omega^{\delta} := \{ x\in \Omega \, : \, d(x,\Omega) > \delta\}\,.
	\end{equation*}
	Since $\Omega$ is connected and satisfies the ball condition with radius $1$, we also obtain that $\Omega^\delta$ is connected for $\delta \in (0,1/2]$ and thus any two points in $\Omega^\delta$ can be connected by a Harnack chain of balls of size $\delta/4$ and with length depending only on $n,p,\diam \Omega$, and $\delta$. By using \Cref{FChainMeas} and \Cref{tr} we then find positive constant $c^\filledstar$ and $t^\filledstar$, both depending only on $n,p,\diam \Omega$, and $r_1$, such that
	\begin{equation*}
		\inf_{x \in \Omega^{1/2}} u(x,t^\filledstar) \geq 1/c^\filledstar\,.
	\end{equation*}
	\Cref{ComparisonFunction} then proves the result whenever $x \in \Omega^1$. Next, let $y \in \Omega \setminus \Omega^1$ and let $y^\filledstar \in \partial \Omega$ be such that $d(y,\partial \Omega)=|y-y^\filledstar|$.
	Since $d(y,\partial \Omega) \leq 1$ and since the direction is unique (see \Cref{ballNTAcork}) we have that $y = a_{2d(y,\partial \Omega)}(y^\filledstar)$. With this at hand we can consider the point $a_{2}(y^\filledstar)$ (which is collinear with $y,y^\filledstar$) satisfying
	\begin{equation*}
		 \fint_{B_{1/4}(a_2(y^\filledstar))} u(x,t^\filledstar) \dx \geq 1/c^\filledstar \,.
	\end{equation*}
	Applying \Cref{ComparisonFunction} we see that
	\begin{equation*}
		u(x,t) \geq \frac{1}{c^\filledstar c_1} \left ( c_1(p-2)[c^\filledstar]^{p-2} (t-t^\filledstar) +1 \right )^{- \frac{1}{p-2}} d(x,\partial B_{1}(a_2(y^\filledstar)))\,,
	\end{equation*}
	whenever $x \in B_{1}(a_2(y^\filledstar))$, $t>t^\filledstar$. Applying this for $x=y$ completes the proof.
\end{proof}
\begin{remark}
	Note that our tools are too rough to obtain the lower bound in \Cref{SUniversalLowerBound} independent on the distribution of the initial data. To remedy this, we assume that the initial data is positive in a region away from the boundary.
\end{remark}

In the next theorem we combine the results of~\Cref{SUniversalUpperBound,SUniversalLowerBound} to obtain an elliptic type global Harnack estimate.

\begin{theorem}
	\label{theo1} Let $\Omega\subset\R^n$ be a bounded $C^{1,1}$-domain satisfying the ball condition with radius $r_0$. Let $u \in C ([0,\infty);L^2(\Omega)) \cap L^{p}(0,\infty; W_{0}^{1,p}(\Omega))$ be a non-negative $p$-parabolic function in $\Omega \times (0,\infty)$. Let
	\begin{equation*}
		\bar\Lambda := \fint_{\Omega} u(x,0) \dx\,.
	\end{equation*}
	Assume that $\supp u(\cdot,0) \subset \Omega^\delta = \{x \in \Omega: d(x,\partial \Omega) > \delta \}$. Then there are constants $c_i \equiv c_i(\frac{\diam \Omega}{r_0},\frac{\delta}{r_0},n,p)$, $i \in \{1,2\}$, such that
	\begin{align*}
		{c_1}^{-1} \leq \frac{u(x,t+\epsilon \bar\Lambda^{2-p})}{u(x,t)} \leq c_1\,,
	\end{align*}
	whenever $\epsilon \in [0,1]$, $x \in \Omega$ and $t \geq c_2 \bar\Lambda^{2-p}$. Furthermore, constants $c_i$, $i \in \{1,2\}$, are stable as $p \to 2^+$.
\end{theorem}

\begin{proof}
	In the following we will use the constants $C_i$, $i \in \{1,\ldots,4\}$, introduced in \Cref{SUniversalUpperBound} and \Cref{SUniversalLowerBound}. Let $t_0 = \max\{C_2,C_4\}\bar\Lambda^{2-p}$ and consider $t \geq t_0 $ and $\epsilon \in (0,1)$. Then, using \Cref{SUniversalUpperBound} and \Cref{SUniversalLowerBound} we see that
	\begin{align*}
		\frac{u(x,t+\epsilon \bar\Lambda^{2-p})}{u(x,t)} &\leq C_1 C_3\left (\frac{(p-2)C_1^{-1} ( t\bar\Lambda^{p-2} + \epsilon) + 1 }{(p-2)C_2 t \bar\Lambda^{p-2}+1} \right )^{-\frac{1}{p-2}}, \\
		\frac{u(x,t+\epsilon \bar\Lambda^{2-p})}{u(x,t)} &\geq \frac{1}{C_1 C_3} \left (\frac{(p-2)C_3 ( t\bar\Lambda^{p-2} + \epsilon) +1}{(p-2)C_1^{-1} t \bar\Lambda^{p-2}+1} \right )^{-\frac{1}{p-2}}.
	\end{align*}
	\Cref{theo2} follows from this by elementary manipulations. We omit further details.
\end{proof}

In the next theorem we use~\Cref{SUniversalUpperBound,SUniversalLowerBound} together with $C^{1,\alpha}$ estimates for weak solution to obtain a global boundary Harnack principle as well as H\"older continuity of ratios of solutions. The intrinsic time interval ensures that the estimate is $p$-stable.

\begin{theorem}
	\label{theo2} Let $\Omega\subset\R^n$ be a bounded $C^{1,1}$-domain satisfying the ball condition with radius $r_0$. Let $u,v \in C([0,\infty);L^2(\Omega)) \cap L^{p}(0,\infty; W_{0}^{1,p}(\Omega))$ be non-negative $p$-parabolic functions in $\Omega \times (0,\infty)$. Let
	\begin{equation*}
		\bar\Lambda_u = \fint_{\Omega} u(x,0) \dx\,, \qquad \bar\Lambda_v = \fint_{\Omega} v(x,0) \dx\,.
	\end{equation*}
	Assume that the initial data is distributed as follows
	\begin{equation*}
		\supp u(\cdot,0)\,,\ \supp v(\cdot,0) \subset \Omega^\delta = \{x \in \Omega: d(x,\partial \Omega) > \delta \}\,.
	\end{equation*}
	Then there exists $\bar C_1 \equiv \bar C_1(\frac{\diam \Omega}{r_0},\frac{\delta}{r_0},n,p)$, such that if $\bar C_1 \leq T_- \leq T_+$ satisfy
	\begin{equation} \label{T+T-}
		T_{-} \min \{\bar\Lambda_u,\bar\Lambda_v\}^{2-p}\leq T_{+} \max \{\bar\Lambda_u,\bar\Lambda_v\}^{2-p}\,,
	\end{equation}
	the following holds. There exist
	$\bar C_i \equiv \bar C_i(\frac{\diam \Omega}{r_0},\frac{\delta}{r_0},T_-,T_+,n,p)$, $i \in \{2,3\}$, such that
	\begin{equation}
		\label{s1} \bar C_2^{-1}\frac{\bar\Lambda_u}{\bar\Lambda_v}\leq \frac{u(x,t)}{v(x,t)} \leq \bar C_2 \frac{\bar\Lambda_u}{\bar\Lambda_v}\,,
	\end{equation}
	whenever
	\begin{equation*}
		(x,t) \in D := \Omega \times (T_{-} \min \{\bar\Lambda_u,\bar\Lambda_v\}^{2-p}\,,\,T_{+} \max \{\bar\Lambda_u,\bar\Lambda_v\}^{2-p})\,.
	\end{equation*}
	Furthermore, there exists an exponent $\sigma \equiv \sigma(n,p)\in (0,1)$, such that
	\begin{equation}
		\label{s2} \left \lvert \frac{u(x,t)}{v(x,t)} - \frac{u(y,s)}{v(y,s)} \right \rvert \leq \bar C_3 \frac{\bar\Lambda_u}{\bar\Lambda_v} \left (|x-y| + \max \{\bar\Lambda_u, \bar\Lambda_v\}^{\frac{2-p}{p}}|t-s|^{1/p} \right)^\sigma
	\end{equation}
	whenever $(x,t), (y,s) \in D$. The constants $\bar C_i$, $i \in \{1,2,3\}$, and $\sigma$ are stable as $p \to 2^+$.
\end{theorem}
\begin{proof}
	In the following we will let $C_1,C_2,C_3$ and $C_4$ be the constants in \Cref{SUniversalUpperBound,SUniversalLowerBound}.
	We begin by proving \Cref{s1}. Indeed, using \Cref{SUniversalUpperBound,SUniversalLowerBound} we see that
	\begin{align} \label{global bhi lower}
		\frac{\bar \Lambda_u}{\bar \Lambda_v}\frac{1}{C_1C_3 }\left (\frac{(p-2)C_1^{-1}\bar \Lambda_v^{p-2} t+1}{(p-2)C_3\bar \Lambda_u^{p-2} t +1} \right )^{\frac{1}{p-2}} \leq \frac{u(x,t)}{v(x,t)}\,,
	\end{align}
	and
	\begin{align} \label{global bhi upper}
		\frac{u(x,t)}{v(x,t)} \leq \frac{\bar \Lambda_u}{\bar \Lambda_v} C_1 C_3 \left (\frac{(p-2)C_3 \bar \Lambda_v^{p-2} t +1} {(p-2)C_1^{-1}\bar \Lambda_u^{p-2} t +1}\right )^{\frac{1}{p-2}}\,,
	\end{align}
	whenever $t \geq \max\{C_2,C_4\} \min\{\bar \Lambda_u,\bar \Lambda_v\}^{2-p}$. In particular we see that if $T_- > \max\{C_2,C_4\}$, and if $T_-,T_+$ satisfy \Cref{T+T-}, then \Cref{s1} holds with a constant $\bar C_2$ depending only on $\frac{\diam \Omega}{r_0},\frac{\delta}{r_0},T_-,T_+,n,p$.
	
	To proceed, consider the rescaled $p$-parabolic functions
	\begin{equation*}
		\tilde u(x,t) = \frac{u(x,\bar \Lambda_u^{2-p} t)}{\bar \Lambda_u}\,, \qquad \tilde v(x,t) = \frac{v(x,\bar \Lambda_v^{2-p} t)}{\bar \Lambda_v}\,.
	\end{equation*}
	Using \Cref{SUniversalUpperBound1} for $w \in \{\tilde u, \tilde v\}$ we get for $t > c_2$ that
	\begin{align*}
		\sup_{x \in \Omega} w(x,t) \leq c_1 \left ( {(p-2)c_1^{-1} t+1}\right )^{-\frac{1}{p-2}} \leq c_1\,,
	\end{align*}
	where $c_1,c_2$ are as in \Cref{SUniversalUpperBound1}.
	Thus we can apply \cite[Theorem 0.1]{L} to conclude that there exist $C$ and $\sigma$, depending on $\Omega,p$ and $n$, such that, for $w \in \{\tilde u, \tilde v\}$,
	\begin{equation*}
		|\grad w(x,t) - \grad w(\tilde x,\tilde t)| \leq C(|x-\tilde x| + |t-\tilde t |^{1/p})^\sigma\,,
	\end{equation*}
	whenever $(x,t), (\tilde x,\tilde t) \in \Omega \times (c_2,\infty)$. In particular, arguing as in \cite[(3.31), p. 2717]{KMN}, we have, for $w \in \{\tilde u, \tilde v\}$,
	\begin{equation}\label{wd tilde holder}
		\left \lvert \frac{w(x,t)}{d(x,\partial \Omega)} - \frac{ w(\tilde x,\tilde t)}{d(\tilde x,\partial\Omega)} \right \rvert \leq C(|x-\tilde x| + |t-\tilde t|^{1/p})^\sigma\,,
	\end{equation}
	whenever $(x,t), (\tilde x,\tilde t) \in \Omega \times (c_2,\infty)$. Scaling back to the original $u,v$, \Cref{wd tilde holder} becomes, with $w \in \{u, v\}$,
	\begin{equation} \label{wd holder}
		\left \lvert \frac{w(x,t)}{d(x,\partial \Omega)} - \frac{ w(\tilde x, \tilde t)}{d(\tilde x,\partial\Omega)} \right \rvert \leq \bar \Lambda_w C(|x-\tilde x| + \bar \Lambda_w^{\frac{p-2}{p}}|t-\tilde t|^{1/p})^\sigma\,,
	\end{equation}
	whenever $(x,t), (\tilde x,\tilde t) \in \Omega \times (c_2 \bar \Lambda_w^{p-2},\infty)$. Next, using the identity
	\begin{align}
		\frac{u(x,t)}{v(x,t)} - \frac{u(y,s)}{v(y,s)} = &\frac{d(x,\partial \Omega)}{v(x,t)} \left ( \frac{u(x,t)}{d(x,\partial \Omega)}-\frac{u(y,s)}{d(y,\partial \Omega)} \right ) \notag \\
		&+ \frac{u(y,s)}{v(y,s)} \frac{d(x,\partial \Omega)}{v(x,t)} \left ( \frac{v(y,s)}{d(y,\partial \Omega)}-\frac{v(x,t)}{d(x,\partial \Omega)} \right )\,, \label{smartidentity}
	\end{align}
	assuming that $s,t \in (T_{-} \min \{\bar\Lambda_u,\bar\Lambda_v\}^{2-p}, T_{+} \max \{\bar\Lambda_u,\bar\Lambda_v\}^{2-p})$ and that $T_{-} > c_2$ (where $c_2$ is from \Cref{SUniversalUpperBound1}), we can apply  \Cref{SUniversalUpperBound,SUniversalLowerBound} together with \Cref{s1,wd holder} in the identity \Cref{smartidentity} to obtain \Cref{s2}. This completes the proof of \Cref{theo2}.
\end{proof}

\begin{remark}
	Considering the estimates \Cref{global bhi lower,global bhi upper} in the proof of \Cref{theo2}, we see that the nonlinearities dominate for large values of $t$. In particular, there exists a $p$-instable constant $C \equiv C(\frac{\diam \Omega}{r_0},\frac{\delta}{r_0},n,p)$ such that
	\begin{equation*}
		C^{-1}\leq \lim_{t\to\infty}\frac{u(x,t)}{v(x,t)} \leq C\,,
	\end{equation*}
	whenever $x\in \Omega$. Note that $C$ is independent of the initial data. This does not happen when $p=2$ and the effect is purely non-linear.
\end{remark}

\begin{remark}
	Note that to prove \Cref{theo2} we rely on estimates established in \Cref{SUniversalUpperBound,SUniversalLowerBound}, instead of relying on the comparison principle, the Carleson estimate and the Harnack inequality as in \cite[Theorem 2.1]{FGS}. This is why our estimates from below depend on the distribution of the initial data. Furthermore, this falls fairly short of the result in \cite{FGS}, but none the less we provide a $p$-stable version of the phenomena involved in our case.
\end{remark}


\section{Local estimates in $C^{1,1}$-domains } \label{s.local}

In this section the main focus is to develop an intrinsic version of the boundary Harnack principle~\Cref{st1}, see \Cref{ss.localbhi}. To do this, we first prove an upper and a lower decay rate estimate in the next section.


\subsection{Upper and lower bound on the decay} \label{ss.locdecay}
We begin with the upper bound, which follows by combining the barrier function from \Cref{UpperBarrier} together with the Carleson estimate \Cref{carleson:thm1}. In the following $M$ will denote the NTA-constant of the $C^{1,1}$ domain $\Omega$, see \Cref{ballNTAcork}. 
\begin{lemma}
	\label{AGSUpbd} Let $u$ be a non-negative solution in $\Omega_T$, where $\Omega$ is a $C^{1,1}$ domain satisfying the ball condition with radius $r_0$. Let $x_0 \in \partial \Omega$ and $0<r \leq r_0$. Let $0< \delta \leq \tilde \delta \leq 1$. Assume that $u(a_r(x_0),t_0) >0$ for a fixed $t_0 \in (0,T)$ and let
	\begin{equation*}
		\tau = \frac{C_4}{16} \left[C_5 u(a_r(x_0),t_0) \right]^{2-p} r^p \,,
	\end{equation*}
	where $C_4$ and $C_5$, both depending on $p,n$, are as in~\Cref{NTA BChain} and with $t_0 > 5 \tilde \delta^{p-1} \tau$. Assume furthermore that $u$ vanishes continuously on $S_T \cap \bigl (B_r(x_0) \times (t_0-4\tilde \delta^{p-1} \tau, t_0-\delta^{p-1} \tau)\bigr )$. Then there exist constants $c_i \equiv c_i (M,p,n)$, $i \in \{8,9\}$, such that
	\begin{equation*}
		\sup_{Q} u \leq \left( c_8 / \tilde \delta \right)^{c_9/ \delta} \frac{d(x_0,\partial \Omega)}{r} u(a_r(x_0),t) \,,
	\end{equation*}
	where $Q := (B_r(x_0) \cap \Omega) \times ( t_0 - 2\tilde \delta^{p-1}\tau, t_0 - \delta^{p-1} \tau)$. Furthermore, constants $c_i$, $i \in \{8,9\}$, are stable as $p \to 2^+$
\end{lemma}
\begin{proof}
	Without loss of generality, we may after scaling assume that $u(a_r(x_0),t_0) = 1$ and $r=4$. Applying \Cref{UpperBarrier,carleson:thm1}, with
	\begin{equation*}
		k := \min\left\{ \frac{p-1}{n} , \tilde \delta \left( \frac{C_4}{16}C_5^{2-p} \right)^{\frac1{p-1}} ,1 \right\}\,, \qquad H := \left( c_6 / \tilde \delta \right)^{c_7/ \delta}\,,
	\end{equation*}
	with constants as in \Cref{carleson:thm1} we get
	\begin{equation*}
		u(x,t) \leq H
	\end{equation*}
	whenever $(x,t) \in (B_{4}(x_0) \cap \Omega) \times (t_0-3 \tilde \delta^{p-1} \tau, t_0-\delta^{p-1} \tau)$.
	The comparison function indexed by its initial time $s_0$ and center point $y_0$, is
	\begin{equation*}
		\hat h_{y_0,s_0}(x,t) = \tilde h(x-y_0,t-s_0)
	\end{equation*}
	with $T = \tilde \delta^{p-1} \tau$, where $\tilde h$ is from \Cref{UpperBarrier}. Let $y \in \partial \Omega \cap B_1(x_0)$ and consider $y_0 = a_2(y)$,  an outer corkscrew point as in \Cref{ballNTAcork}. Then,  by the comparison principle
	\begin{equation} \label{eq:AGSupbdcomp}
		u(x,t) \leq \hat h_{y_0,s_0}(x,t)\,,
	\end{equation}
	in
	\begin{equation*}
		(\Omega \cap [B_{1+k}(y_0) \setminus B_1(y_0)]) \times (s_0, s_0+T)\,,
	\end{equation*}
	whenever
	\begin{equation*}
		(s_0,s_0+T) \subset (t_0 - 3 \tilde \delta^{p-1} \tau, t_0 - \delta^{p-1}\tau)\,.
	\end{equation*}
	From \Cref{UpperBarrier,eq:AGSupbdcomp} we have the estimate
	\begin{equation*}
		u(x,t) \leq \frac{H \exp(2)}{k} (|y_0-x|-1)\,,
	\end{equation*}
	whenever
	\begin{equation*}
		(x,t) \in (\Omega \cap [B_{1+k}(y_0) \setminus B_1(y_0)]) \times (t_0-2 \tilde \delta^{p-1} \tau, t_0-\delta^{p-1} \tau)\,,
	\end{equation*}
	and $y_0 \in \partial \Omega \cap B_1(x_0)$. From this the result follows by scaling back.
\end{proof}
The following lemma establishes to local lower bound on the decay, by combining the barrier from \Cref{lembar} and the Harnack estimates in \Cref{NTA FChain}.
\begin{lemma}
	\label{th2-} Let $u$ be a non-negative solution in $\Omega_T$, where $\Omega$ is a $C^{1,1}$ domain satisfying the ball condition with radius $r_0$. Let $x_0 \in \partial \Omega$ and let $0 < r < r_0$ be fixed. Let $A^- = (a_{r}(x_0),t_0)$, $\theta_- = u(A^-)^{2-p}$ and $t_0 \in (0,T)$. There exists constants $c_i \equiv c_i(M,p,n)$, $i \in \{3,4\}$, such that if
	\begin{equation*}
		\theta_- r^p < t_0\,, \qquad \text{and} \qquad t_0 + 2 c_4 \theta_- r^p < T\,,
	\end{equation*}
	then
	\begin{equation*}
		\frac{1}{c_3} \frac{d(x, \partial \Omega)}{r}u(A^-) \leq u(x,t)
	\end{equation*}
	for $x \in B_{r}(x_0) \cap \Omega$ and $t_0+c_4 \theta_- r^p < t < t_0+2 c_4 \theta_- r^p$. Furthermore, constants $c_i$, $i \in \{3,4\}$, are stable as $p \to 2^+$.
\end{lemma}
\begin{proof}
	Set $\lambda_- := u(A^-) $ and consider the scaled solution
	\begin{equation*}
		v(y,s) = \frac{1}{\lambda_-}u(x_0 + r y,t_0+s \lambda_-^{2-p} r^p) \,.
	\end{equation*}
	Set also $\tilde \Omega := \{ y \, : \, x_0 + ry \in \Omega \}$ so that $0 \in \partial \tilde \Omega$. For the new function $v$ we have the following situation, denoting $A_v^- = ((a_r(x_0)-x_0)/r,0)$,
	\begin{align*}
		v(A_v^-) &= 1\,, \\
		d(A^-, \partial \tilde \Omega) &= 1\,,
	\end{align*}
	and $v$ is a solution in $\tilde \Omega \times (-1,\tilde T)$, where $\tilde T := (T-t_0)\lambda_-^{p-2} r^{-p}$. Since $0 < r < r_0/4$ we know that $\tilde \Omega$ satisfies the ball condition with radius $4$. To continue, consider the following set
	\begin{equation*}
		D = \{y \in \tilde \Omega \, : \, d(y,B_1(0)\cap \partial \tilde \Omega) = d(y, \partial \tilde \Omega) = 1\}\,.
	\end{equation*}
	Note that $D \subset B_2(0) \cap \tilde \Omega$ and that $\sup \{d(a_1(y_0),y): y \in D\} \leq 2$ for any $y_0 \in \partial \tilde \Omega \cap B_1(0)$. We obtain from the Harnack chain estimate in~\Cref{NTA FChain} (applied with $\delta = \min\{c_h^{(2-p)/p},1\}$ where $c_h$ is from \Cref{NTA FChain}), that there is a $\tilde \tau_1 > 0$ depending only on $n,p,M$ such that
	\begin{equation*}
		v(x,\tilde \tau_1) \geq \frac{1}{\tilde c_1}\,,
	\end{equation*}
	whenever $x \in \{y: d(y,D) < 1/4 \}$ provided $\tilde T > \tilde \tau_1$. Using ~\Cref{ComparisonFunction} (applied with $r=1$, $\varrho=1/4$, $g = v(\cdot,\tilde \tau_1)$, $x_0 = \tilde y \in D$, $t_0 = \tilde \tau_1$) for all points $\tilde y \in D$, we get
	\begin{equation}
		\label{ComparisonLowerBound} v(y,t) \geq \frac{1}{\tilde c_1 c_1} \left( c_1 \tilde c_1^{2-p} (p-2) (t-\tilde \tau_2) + c_1 c_2 (p-2) + 1 \right )^{-\frac{1}{p-2}} d(y, \tilde \Omega)
	\end{equation}
	whenever $(y,t) \in ( \tilde \Omega \cap B_1(0) ) \times (\tilde \tau_2,\tilde T)$, with $\tilde \tau_2 = \tilde \tau_1 + c_2 \tilde c_1^{p-2}$ provided $\tilde T > \tilde \tau_2$. Going back to $\Omega$ and $u$ gives us the result provided $\tilde T > 2 \tilde \tau_2$.
\end{proof}

Combining \Cref{AGSUpbd,th2-} we obtain the joint estimate.

\begin{theorem}
	\label{th2} Let $u$ be a non-negative solution in $\Omega_T$, where $\Omega$ is a $C^{1,1}$ domain satisfying the ball condition with radius $r_0$. Let $x_0 \in \partial \Omega$, $t_0 \in (0,T)$, and let $0 < r < r_0$ be fixed. Let $A^- = (a_{r}(x_0),t_0)$ and $\theta_- = u(A^-)^{2-p}$. There exists constants $c_i \equiv c_i(M,p,n)$, $i \in \{5,6\}$, such that if
	\begin{equation*}
		\theta_- r^p < t_0\,, \qquad \text{and} \qquad t_0 + 2 c_4 \theta_- r^p < T\,,
	\end{equation*}
	then for $A^+ = (a_{r}(x_0),t_0 + 2 c_4 \theta_- r^p)$ and $\theta_+ = c_6^{-1} u(A^+)^{2-p}$ (where $c_4$ is from \Cref{th2-}), we have
	\begin{equation*}
		5 \theta_+ \leq \theta_-\,.
	\end{equation*}
	Furthermore, if $u$ vanishes continuously on
	\begin{equation*}
		S_T\cap \bigl (B_r(x_0) \times \left (t_0+\big [2c_4\theta_- -5\theta_+ \big ] r^p,t_0+\big [2 c_4\theta_- - \theta_+ \big ] r^p \right )\bigr )\,,
	\end{equation*}
	then
	\begin{equation*}
		\frac{1}{c_5} \frac{d(x, \partial \Omega)}{r}u(A^-) \leq u(x,t) \leq c_5 \frac{d(x, \partial \Omega)}{r} u(A^+)
	\end{equation*}
	for $x \in B_{r}(x_0) \cap \Omega$ and $t_0+\big [2 c_4 \theta_- - 2\theta_+\big ] r^p < t < t_0+\big [2 c_4 \theta_- - \theta_+\big ] r^p$. Furthermore, constants $c_i$, $i \in \{5,6\}$, are stable as $p \to 2^+$.
\end{theorem}
\begin{proof}
	Rescale $u$ as in the proof of \Cref{th2-}, let also $\tilde \tau_2$ be as in the proof of \Cref{th2-}. Thus
	\begin{equation*}
		\frac{1}{c_3} d(x, \partial \Omega) \leq v(x,t)
	\end{equation*}
	holds for $(x,t)  \bigl (B_{r}(x_0) \cap \Omega) \times (\tilde \tau_2, 2 \tilde \tau_2)$.
	Define $\tau_+ = 2\tilde \tau_2$ and consider $A_+ = (a_1(0),\tau_+)$. Then using \cref{ComparisonLowerBound} we get
	\begin{align*}
		v(A_+)^{2-p} \leq & \left [\frac{1}{\tilde c_1 c_1} \left( c_1 \tilde c_1^{2-p} (p-2) (t-\tilde \tau_2) + c_1 c_2 (p-2) + 1 \right )^{-\frac{1}{p-2}} \right ]^{2-p} \\
		= & \frac{1}{(\tilde c_1 c_1)^{2-p}} \left [ c_1 \tilde c_1^{2-p} (p-2) + c_1 c_2 (p-2) + 1 \right ] \\
		=: & \tilde c_2\,.
	\end{align*}
	We will now apply \Cref{AGSUpbd} with ($r=1$, $\delta^{p-1} = \min\{\frac{16}{5 \tilde c_2 C_4} C_5^{p-2},1\}$ and $\tilde \delta = \delta$, with $C_4,C_5$ from \Cref{AGSUpbd}). Doing this we see that
	\begin{align*}
		v(x,t) \leq \tilde c_3 d(x,\partial \tilde \Omega) v(A_+)\,,
	\end{align*}
	whenever $(x,t) \in (B_{1}(0) \cap \Omega) \times \big (\tau_+-v(A_+)^{2-p},\tau_+-v(A_+)^{2-p}/2 \big )$. Going back to $\Omega$ and $u$ gives us the result.
\end{proof}


\subsection{Local boundary Harnack estimate}
\label{ss.localbhi}
We are now ready to state and prove our local boundary Harnack principle, consult \Cref{figBHI} for a schematic of the geometry.
\begin{theorem}
	\label{theo3+} Let $u,v$ be two non-negative solutions in $\Omega_T$, where $\Omega$ is a $C^{1,1}$-domain satisfying the ball condition with radius $r_0$. Let $x_0 \in \partial \Omega$ $t_0 \in (0,T)$, and $0 < r < r_0$ be fixed. Let $A_- = (a_{r}(x_0),t_0)$ and assume that $u(A_-)=v(A_-)$. Let the constants $c_i$, $i \in \{4,5,6\}$ be as in \Cref{th2-,th2}. Let $\theta_- = u(A_-)^{2-p}$, and assume
	\begin{equation*}
		\theta_- r^p < t_0\,, \qquad \text{and} \qquad t_0 + 2 c_4 \theta_- r^p < T\,.
	\end{equation*}
	Set
	\begin{align*}
		A_+ &= (a_{r}(x_0),t_0 + 2c_4 \theta_- r^p)\,, &\theta_{+,u} = c_6^{-1} u(A_+)^{2-p}\,.
	\end{align*}
	Assume that $v(A_+) \geq u(A_+)$. Then there exists a time $t_+^\filledstar$, depending on $v$, satisfying
	\begin{align*}
		t_+^\filledstar &\in (t_0+( 2 c_4 \theta_- - \theta_{+,u})r^p, t_0+2 c_4 \theta_- r^p) \\
		A_+^\filledstar &= (a_{r}(x_0),t_+^\filledstar) &\theta_{+,v}^\filledstar = c_6^{-1} v(A_+^\filledstar)^{2-p},
	\end{align*}
	such that the following holds. If both $u$ and $v$ vanish continuously on
	\begin{equation*}
		S_T\cap\bigl (B_r(x_0) \times \left (t_0+\big [2c_4\theta_- -5\theta_{+,u} \big ] r^p,t_0+\big [2 c_4\theta_- - \theta_{+,u} \big ] r^p \right )\bigr )\,,
	\end{equation*}
	then
	\begin{equation*}
		\frac{1}{c_5^2} \frac{u(A_-)}{v(A_+^\filledstar)} \leq \frac{u(x,t)}{v(x,t)} \leq c_5^2 \frac{u(A_+)}{v(A_-)}\,,
	\end{equation*}
	whenever $(x,t)$ belongs to the set
	\begin{equation*}
		(B_{r}(x_0) \cap \Omega) \times \big (t_0+\big [2 c_4 \theta_- - ( \theta^\filledstar_{+,v} + \theta_{+,u}) \big] r^p,\, t_0+\big[2 c_4 \theta_- - \theta_{+,u} \big ]r^p \big ) \,.
	\end{equation*}
\end{theorem}
\begin{figure}[h]
	\begin{center}
		\begin{tikzpicture}[scale=1.5]
			\draw [->] (0,1)--(0,5);
			\node[right] at (0.1,4.8) {$t$};
			\draw[pattern=north west lines] (0,3) rectangle (4,3.5);
			\draw[pattern=crosshatch] (0,3.5) rectangle (4,4);
			\node[fill=white,inner sep=1pt] at (2,3.75) {${(x,t)}$};
			\draw[dashed] (4,3)--(4,1);
			\draw[dashed] (4,4)--(4,5);
			\draw[fill=black] (4,5) circle [radius=0.05];
			\node[left] at (4,5) {$A_+$};
			\draw[fill=black] (4,4.5) circle [radius=0.05];
			\node[left] at (4,4.5) {$A_+^\filledstar$};
			\draw[fill=black] (4,1) circle [radius=0.05];
			\node[left] at (4,1) {$A_-$};
			\draw (4.1,5)--(4.2,5)--(4.2,1)--(4.1,1);
			\node[right] at (4.2,1.5) { $2 c_4 \theta_{-} r^p$};
			\draw (4.3,5)--(4.4,5)--(4.4,4)--(4.3,4);
			\node[right] at (4.4,4.5) {$\theta_{+,u} r^p$};
			\draw (4.3,4)--(4.4,4)--(4.4,3)--(4.3,3);
			\node[right] at (4.4,3.25) {$\theta_{+,u} r^p$};
			\draw (4.5,4)--(4.6,4)--(4.6,3.5)--(4.5,3.5);
			\node[right] at (4.6,3.75) {$\theta_{+,v}^\filledstar r^p$};
		\end{tikzpicture}
		\caption[Denotes the region where the right-hand-side/left-hand-side of \Cref{theo3+} holds respectively.]{The boxes (
			\tikz \filldraw[pattern=north west lines] (0ex,0ex) rectangle (2ex,2ex);
			\tikz \filldraw[pattern=crosshatch] (0ex,0ex) rectangle (2ex,2ex);
			) denotes the region where the right-hand-side/left-hand-side of \Cref{theo3+} holds respectively.
		} \label{figBHI}
	\end{center}
\end{figure}
\begin{remark}
	It should be noted that we cannot control the time $t_+^\filledstar$ except which interval it lies in, it is a purely intrinsic parameter. Furthermore note that \Cref{theo3+} is equivalent to the boundary Harnack principle \Cref{st1} when $p=2$.
\end{remark}
\begin{proof}
	Let $c_i$, $i \in \{3,\ldots,6\}$ be as in~\Cref{th2-,th2}. By the assumptions, we know that $\theta_{+,u} \geq \theta_{+,v} := c_6^{-1} v(A_+)^{2-p}$. We then obtain by~\Cref{th2}, for $x \in B_r(x_0) \cap \Omega$ and $t_0+ (2 c_4 \theta_- - 2\theta_{+,u}) r^p < t < t_0+(2 c_4 \theta_- - \theta_{+,u}) r^p$, that
	\begin{equation}
		\label{linlowup1u} \frac{1}{c_5} \frac{d(x,\partial \Omega)}{r}u(A_-) \leq u(x,t) \leq c_5 \frac{d(x,\partial \Omega)}{r} u(A_+)\,.
	\end{equation}
	Using \Cref{th2-} for $x \in B_r(x_0) \cap \Omega$ and $t_0+c_4 \theta_- r^p < t < t_0+2 c_4 \theta_-r^p$ we get that
	\begin{equation}
		\label{lowbd1} \frac{1}{c_3} \frac{d(x,\partial \Omega)}{r}v(A_-) \leq v(x,t)\,.
	\end{equation}
Now let $t_+ = t_0 + 2 c_4 \theta_- r^p$, and let $t_+^\filledstar$, a time to be fixed, such that $t_+-\theta_{+,u} r^p < t_+^\filledstar \leq t_+$. First note that if $t_+^\filledstar=t_+$ we have for $\theta_{+,v}^\filledstar = c_6^{-1} (v(a_{r}(x_0),t_+^\filledstar))^{2-p}$,
	\begin{equation*}
		t_+^\filledstar - \theta_{+,v}^\filledstar r^p \geq t_+ - \theta_{+,u} r^p \,,
	\end{equation*}
	furthermore if $t_+^\filledstar = t_+-\theta_+^u r^p$ then
	\begin{equation*}
		t_+^\filledstar - \theta_{+,v}^\filledstar r^p < t_+ - \theta_{+,u} r^p \,.
	\end{equation*}
	Thus by continuity there is a largest $t_+^\filledstar$ such that
	\begin{equation*}
		\label{tastdef} t_+^\filledstar - \theta_{+,v}^\filledstar r^p = t_+ - \theta_{+,u} r^p\,.
	\end{equation*}
	With $t_+^\filledstar$ at hand we now apply \Cref{AGSUpbd} (with the same $\delta,\tilde \delta$ as in the proof of \Cref{th2}) combining it with \Cref{lowbd1} to get
	\begin{equation}
		\label{linlowup1v} \frac{1}{c_5} \frac{d(x,\partial \Omega)}{r}v(A_-) \leq v(x,t) \leq c_5 \frac{d(x,\partial \Omega)}{r} v(A_+^\filledstar)
	\end{equation}
	for $x \in B_r(x_0) \cap \Omega$ and $t_0+(2 c_4 \theta_{-} - (\theta_{+,u}+\theta_{+,v}^\filledstar)) r^p < t < t_0+(2 c_4 \theta_- - \theta_{+,u}) r^p$. Combining \Cref{linlowup1u,linlowup1v} we have completed the proof.
\end{proof}


\subsection{Boundary measures in $C^{1,1}$-domains}

We conclude the section by describing the fine properties of the boundary measure defined in~\cref{1.1+}. The first theorem tells that the induced measure is mutually absolutely continuous with respect to the surface measure of $S_T$.

\begin{theorem}
	\label{thmabscont} Under the hypothesis of~\Cref{th2},
	\begin{equation*}
		0 < \liminf_{\varrho \to 0} \frac{\mu_u \big(Q_\varrho(x,t) \big)}{\varrho^{n+1}} \leq \limsup_{\varrho \to 0} \frac{\mu_u \big(Q_\varrho(x,t) \big)}{\varrho^{n+1}} < + \infty
	\end{equation*}
	where $Q_\varrho(x,t) := B_\varrho(x) \times (t-\varrho^2,t)$, whenever $(x,t) \in V$,
	\begin{equation*}
		V := (\partial\Omega \cap B_r(x_0)) \times ( t_0+(2 c_4 \theta_- - 2\theta_+) r^p, t_0+(2 c_4 \theta_- - \theta_+) r^p)\,.
	\end{equation*}
	In particular, $\mu_u$ is mutually absolutely continuous with respect to the surface measure of $S_T$ on $V$.
\end{theorem}
\begin{proof}
	By~\Cref{th2} we have that
	\begin{equation}
		\label{lambda_pm} \lambda_- d(x,\partial \Omega) \leq u(x,t) \leq \lambda_+ d(x,\partial \Omega) \,, \qquad \lambda_\pm := c_5^{\pm 1} \frac{u(A_r^\pm) }{r}
	\end{equation}
	whenever $(x,t) \in Q$ with
	\begin{equation*}
		Q: = (\Omega \cap B_r(x_0)) \times \left( t_0+( 2 c_4 \theta_- - 2\theta_+) r^p, t_0+(2 c_4\theta_- - \theta_+) r^p \right)
	\end{equation*}
	and $\theta_\pm$ are as in~\Cref{th2}. We now pick a point $(y,s) \in S_T \cap \partial_p Q$. Choose $\varrho$ small enough so that $U_{\varrho}(y,s) \cap (\Omega \times \R)$ is contained in $Q$, where
	\begin{equation*}
		U_\varrho(y,s) := B_{\varrho} (y) \times (s- \tilde \tau \varrho^2, s +\tilde \tau \varrho^2)\,,
	\end{equation*}
	$\tilde \tau := \lambda_-^{2-p} \max\{8 C_4 C_5^{2-p}, 2(\tau_0 + \tau_1) \}$ and $C_4,C_5$ and $\tau_0,\tau_1$ are as in~\Cref{MuUpperBound,MuLowerBound}. After a simple covering argument using~\cref{lambda_pm}, and~\Cref{MuUpperBound,MuLowerBound},  we find a constant $C \equiv C(p,n,M,\lambda_\pm)$ such that
	\begin{equation*}
		\frac1C \leq \frac{\mu_u(U_{\varrho/2}(y,s))}{\varrho^{n+1}} \leq C\,.
	\end{equation*}
	Taking a possibly larger $C$, and a smaller $\varrho$, this actually implies that
	\begin{equation*}
		\frac1C \leq \frac{\mu_u(Q_{\varrho}(y,s))}{\varrho^{n+1}} \leq C\,,
	\end{equation*}
	uniformly for small enough $\varrho$. This proves the statement.
\end{proof}
\begin{remark}
	Note that in the same region $V$ as in \Cref{thmabscont} we have that the measure is doubling. Moreover note that \Cref{th2-} implies a Hopf-type result on this boundary cylinder $V$, thus together with the fact that solutions are $C^{1,\alpha}$ up to the boundary we get that the logarithm of the normal derivative on the boundary is H\"older continuous. Now arguing as in \cite[(1.7)--(1.10)]{ALNopt} we get, for $(x_0,t_0) \in V $ given, and $\epsilon \in (0,1)$, that
	\begin{equation*}
		\lim_{\varrho \to 0} \frac{\mu_u(Q_{\epsilon \varrho}(x_0,t_0))}{\mu_u(Q_{\varrho}(x_0,t_0))} = \epsilon^{n+1}\,.
	\end{equation*}
	In particular, the measure $\mu_u$ is asymptotically optimal doubling.
\end{remark}


\end{document}